\documentclass[a4paper,10pt,reqno]{amsart}
\usepackage{color}
\usepackage[hyphens]{url}
\usepackage[square,numbers,comma,sort&compress]{natbib}
\usepackage{enumerate}
\usepackage{colortbl}
\usepackage{booktabs}
\usepackage{iba-algo}
\usepackage{ifpdf}
\usepackage{graphics}
\usepackage{epsfig}
\usepackage{subfigure}
\usepackage{booktabs}
\usepackage{mathrsfs}   
\def\ve#1{\mathchoice{\mbox{\boldmath$\displaystyle\bf#1$}}
{\mbox{\boldmath$\textstyle\bf#1$}}
{\mbox{\boldmath$\scriptstyle\bf#1$}}
{\mbox{\boldmath$\scriptscriptstyle\bf#1$}}}
\newcommand\vealpha{{\boldsymbol{\alpha}}}

\newcommand\velambda{{\boldsymbol{\lambda}}}

\newcommand\Z{\mathbb Z}   
\newcommand\R{\mathbb R}   
\newcommand\Q{\mathbb Q}   
\newcommand\Po{\mathcal P}   
\newcommand\Co{\mathcal C}   
\newcommand\Tr{\mathcal T}   
\newcommand\In{\mathcal I}  

\DeclareMathOperator{\td}{td}
\DeclareMathOperator{\cone}{cone}
\DeclareMathOperator{\interior}{relint}

\DeclareMathOperator{\conv}{conv}   
\DeclareMathOperator{\vol}{vol}     
\DeclareMathOperator{\supp}{supp}        
\DeclareMathOperator{\affine}{aff}       
\DeclareMathOperator{\rank}{rank}
\DeclareMathOperator{\cl}{cl}       
\DeclareMathOperator{\lin}{lin}     

\let\epsilon=\varepsilon
\usepackage{ifthen}
\makeatletter
\newcommand{\DeclareBracket}[3]{
  \newcommand{#1}[2][]{%
  \ifthenelse%
  {\equal{##1}{}}%
  {\left#2##2\right#3}%
  {\csname ##1l\endcsname#2##2\csname ##1r\endcsname#3}}}    
\makeatother
\DeclareBracket\abs||           
\DeclareBracket\norm\|\|        
\DeclareBracket\floor\lfloor\rfloor
\DeclareBracket\ceil\lceil\rceil
\DeclareBracket\set\{\}
\DeclareBracket\paren()
\DeclareBracket\inner\langle\rangle
\DeclareBracket\fractional\{\}
 
\newcommand\C{\mathbb C}

\usepackage{amsthm}
\newtheorem{theorem}{Theorem}%
\makeatletter
\newtheorem{lemma}{Lemma}
\renewcommand*{\c@lemma}{\c@theorem}
\renewcommand*{\p@lemma}{\p@theorem}

\newtheorem{conjecture}{Conjecture}
\renewcommand*{\c@conjecture}{\c@theorem}
\renewcommand*{\p@conjecture}{\p@theorem}

\renewcommand*{\c@proposition}{\c@theorem}
\renewcommand*{\p@proposition}{\p@theorem}

\newtheorem{corollary}{Corollary}
\renewcommand*{\c@corollary}{\c@theorem}
\renewcommand*{\p@corollary}{\p@theorem}

\renewcommand*{\c@observation}{\c@theorem}
\renewcommand*{\p@observation}{\p@theorem}

\theoremstyle{definition}

\renewcommand*{\c@problem}{\c@theorem}
\renewcommand*{\p@problem}{\p@theorem}

\newtheorem{definition}{Definition}
\renewcommand*{\c@definition}{\c@theorem}
\renewcommand*{\p@definition}{\p@theorem}

\newtheorem{remark}{Remark}
\renewcommand*{\c@remark}{\c@theorem}
\renewcommand*{\p@remark}{\p@theorem}

\newtheorem{example}{Example}
\renewcommand*{\c@example}{\c@theorem}
\renewcommand*{\p@example}{\p@theorem}

\newtheorem{algorithm}{Algorithm}
\renewcommand*{\c@algorithm}{\c@theorem}
\renewcommand*{\p@algorithm}{\p@theorem}

\makeatother

\usepackage[breaklinks=true,colorlinks,citecolor=blue]{hyperref}

\title[Ehrhart Polynomials of Matroid Polytopes and Polymatroids]
{Ehrhart Polynomials of\\ Matroid Polytopes and Polymatroids}
\author{Jes\'us A. De Loera}
\address{Jes\'us A. De Loera: Department of Mathematics,
University of California,
Davis, CA 95616, USA}
\email{deloera@math.ucdavis.edu}

\author{David C. Haws}
\address{David C. Haws: Department of Mathematics,
University of California,
Davis, CA 95616, USA}
\email{haws@math.ucdavis.edu}

\author{Matthias K\"oppe}
\address{Matthias~K\"oppe: Otto-von-Guericke-Universit\"at Magdeburg, Department of
  Mathematics, Institute for Mathematical Optimization (IMO),
  Univer\-si\-t\"ats\-platz~2, 
  39106 Magdeburg, Germany} 
\email{mkoeppe@imo.math.uni-magdeburg.de}

\begin{document}

\date{$\relax$Revision: 1.83 $ - \ $Date: 2007/10/19 21:51:31 $ $}

\begin{abstract}
We investigate properties of Ehrhart polynomials for matroid
polytopes, independence matroid polytopes, and polymatroids. In the
first half of the paper we prove that for fixed rank their Ehrhart
polynomials are computable in polynomial time. The proof relies on the
geometry of these polytopes as well as a new refined analysis of the
evaluation of Todd polynomials. In the second half we discuss two
conjectures about the $h^*$-vector and the coefficients of 
Ehrhart polynomials of matroid polytopes; we provide theoretical and 
computational evidence for their validity.
\end{abstract}

\maketitle

\section{Introduction}
  
Recall that a \emph{matroid} $M$ is a finite
collection of subsets $\mathcal{F}$ of $[n] = \{1,2,\dots,n\}$ called
\emph{independent sets}, such that the following properties are satisfied:
{\bf (1)} $\emptyset \in \mathcal{F}$, {\bf (2)} if $X \in \mathcal{F}$ and $Y \subseteq X$ then $Y \in \mathcal{F}$, {\bf (3)} if $U, V \in \mathcal{F}$ and $|U| = |V| + 1$ there exists $x \in U \setminus V$ such that $V \cup x \in \mathcal{F}$. In this paper
we investigate  convex polyhedra associated with matroids.


One of the reasons matroids have become fundamental objects in pure and applied
combinatorics is their many equivalent axiomatizations. For instance,
for a matroid $M$ on $n$ elements with independent sets $\mathcal F$ the
\emph{rank function} is a function $\varphi\colon 2^{[n]} \rightarrow \Z$ where
$\varphi (A) := \max \{\, |X| \mid X \subseteq A, \, X \in \mathcal F \,\}$.
 Conversely a function $\varphi \colon 2^{[n]} \rightarrow \Z$ is
 the rank function of a matroid on $[n]$ if and only if the
 following are satisfied: {\bf (1)} $0 \leq \varphi (X) \leq |X|$, {\bf (2)} $X \subseteq Y \Longrightarrow \varphi (X) \leq \varphi(Y)$, {\bf (3)} $\varphi (X \cup Y) + \varphi (X \cap Y) \leq \varphi (X) + \varphi (Y)$.
Similarly, recall that a matroid $M$ can be defined by its
\emph{bases}, which are the inclusion-maximal independent sets. The
bases of a matroid $M$ can be recovered by its rank function
$\varphi$. For the reader we recommend \cite{Oxley1992Matroid-Theory}
or \cite{Welsh1976Matroid-Theory} for excellent introductions to the
theory of matroids.


Now we introduce the main object of this paper.  Let $\mathcal{B}$ be the set
of bases of a matroid $M$. If $B = \{\sigma_1,\ldots,\sigma_r\} \in
\mathcal{B}$, we define the \emph{incidence vector of B} as $\ve  e_B :=
\sum_{i=1}^r \ve e_{\sigma_i}$, where $\ve e_j$ is the standard elementary
$j$th vector in $\R^n$. The \emph{matroid polytope} of $M$ is defined as
$\Po(M) := \conv \{\, \ve e_B  \mid B \in \mathcal{B} \, \}$, where
$\conv(\cdot)$ denotes the convex hull. This is different from the well-known
\emph{independence matroid polytope}, $\Po^{\In} (M) := \conv \{ \, \ve e_I  \mid  I
        \subseteq B \in \mathcal B \, \}$, the convex hull of the incidence vectors
of all the independent sets. We can see that $\Po(M) \subseteq \Po^\In(M)$ and
$\Po(M)$ is a face of $\Po^\In(M)$ lying in the hyperplane $\sum_{i=1}^n x_i =
\rank(M)$, where $\rank(M)$ is the cardinality of any basis of $M$. 

\emph{Polymatroids} are closely related to matroid polytopes and
independence matroid polytopes.  We first recall some basic
definitions (see \cite{Welsh1976Matroid-Theory}). A function
$\psi \colon 2^{[n]} \longrightarrow \R$ is \emph{submodular} if $\psi(X \cap
Y) + \psi(X \cup Y) \leq \psi(X) + \psi(Y)$ for all $X, Y \subseteq
[n]$. A function $\psi \colon 2^{[n]} \longrightarrow \R$ is
\emph{non-decreasing} if $\psi(X) \leq \psi(Y)$ for all $X \subseteq Y
\subseteq [n]$. We say $\psi$ is a \emph{polymatroid rank
function} if it is submodular, non-decreasing, and $\psi(\emptyset) =
0$. For example, the rank function of a matroid is a
polymatroid rank function. The \emph{polymatroid} determined by a polymatroid rank function
$\psi$ is the convex polyhedron (see Theorem 18.2.2 in
\cite{Welsh1976Matroid-Theory}) in $\R^n$ given by
\begin{equation*}
    \Po(\psi) := \Big\{ \, \ve x \in \R^n \ \Big| \ \sum_{i \in A} x_i \leq \psi(A) \ \forall A \subseteq [n], \, \ve x \geq \ve0 \, \Big\}  .
\end{equation*}

Independence matroid polytopes are a special class of polymatroids. Indeed, if
$\varphi$ is a rank function on some matroid $M$, then $\Po^\In(M) =
\Po(\varphi)$ \cite{Edmonds2003Submodular-func}. Moreover, the matroid polytope
$\Po(M)$ is the face of $\Po(\varphi)$ lying in the hyperplane $ \sum_{i=1}^n
x_i = \varphi([n]) $.  Matroid polytopes and polymatroids appear in
combinatorial optimization \cite{Schrijver2003Combinatorial-O}, algebraic
combinatorics \cite{Feichtner2004Matroid-polytop}, and algebraic geometry
\cite{Gelfand1987Combinatorial-g}. The main theme of this paper is the study
of the volumes and Ehrhart functions of matroid polytopes, independence matroid
polytopes, and polymatroids (from now on we often refer to all three families
as matroid polytopes).

To state our main results recall that given an integer $k > 0$ and a
polytope $\Po \subseteq \R^n$ we define $k \Po := \{\, k \ve \alpha  \mid  \ve \alpha \in \Po
\, \}$ and the function $i(\Po,k) := \#(k\Po \cap \Z^n) $, where we
define $i(\Po,0) :=1$. It is well known that for integral polytopes,
as in the case of matroid polytopes, $i(\Po,k)$ is a polynomial,
called the \emph{Ehrhart polynomial} of $\Po$.  Moreover the leading
coefficient of the Ehrhart polynomial is the \emph{normalized volume}
of $\Po$, where a unit is the volume of the fundamental domain of the
affine lattice spanned by $\Po$ \cite{Stanley1996Combinatorics-a}. Our
first theorem states:

\begin{theorem} \label{volume}
  Let $r$ be a fixed integer.
  Then there exist algorithms whose input data consists of a number $n$ and an evaluation oracle for 
  \begin{enumerate}[\rm(a)]
  \item a rank function~$\varphi$ of a matroid~$M$ on $n$ elements
    satisfying $\varphi(A) \leq r$ for all~$A$, or
  \item an integral polymatroid rank function~$\psi$ satisfying $\psi(A) \leq r$ for
    all~$A$,
  \end{enumerate}
  which compute in time polynomial in~$n$ the Ehrhart polynomial (in
  particular, the volume) of the matroid polytope $\Po(M)$, the independence
  matroid polytope $\Po^\In(M)$, and the polymatroid~$\Po(\psi)$,
  respectively.
\end{theorem}

The computation of volumes is one of the most fundamental geometric
operations and it has been investigated by several authors from the
algorithmic point of view. While there are a few cases for which the volume
can be computed efficiently (e.g., for convex polytopes in fixed
dimension), it has been proved that computing the volume of polytopes
of varying dimension is $\# P$-hard
\cite{dyerfrieze88,brightwellwinkler91,khachiyan93,lawrence91}. Moreover
it was proved that even approximating the volume is hard
\cite{elekes86}. Clearly, computing Ehrhart polynomials is a harder
problem still.  To our knowledge two previously known families of
varying-dimension polytopes for which there is efficient computation of the
volume are simplices or simple polytopes for which the
number of vertices is polynomially bounded (this follows from Lawrence's
volume formula \cite{lawrence91}). Already for simplices
it is at least NP-hard to compute the whole list of coefficients of the 
Ehrhart polynomial, while recently
\cite{barvinok-2006-ehrhart-quasipolynomial}  presented a
polynomial time algorithm to compute any fixed number of the highest
coefficients of the Ehrhart polynomial of a simplex of varying
dimension.  Theorem \ref{volume} provides another interesting family
of varying dimension with volume and Ehrhart polynomial that can be
computed efficiently.  The proof of Theorem \ref{volume}, presented in
Section \ref{sec:compehrhart}, relies on the geometry of tangent cones
at vertices of our polytopes as well as a new refined analysis of the
evaluation of Todd polynomials in the context of the
computational theory of rational generating functions developed by
\cite{bar,barvinok:99,barvinok-woods-2003,barvinok-2006-ehrhart-quasipolynomial,latte1,latte2,Woods:thesis,verdoolaege-woods-2005}.

In the second part of the paper, developed in Section \ref{sec:hstar},
we investigate algebraic properties of the Ehrhart functions of
matroid polytopes: The \emph{Ehrhart series} of a polytope $\Po$ is
the infinite series $\sum_{k=0}^\infty i(\Po,k)t^k$. We recall the
following classic result about Ehrhart series (see e.g.,
\cite{Hibi1992Algebraic-Combi, Stanley1996Combinatorics-a}). Let 
$\Po \subseteq \R^n$ be an integral convex polytope of dimension
$d$. Then it is known that its Ehrhart series is a rational function
of the form
\begin{equation} \label{lem:ratehr}
\sum_{k=0}^\infty i(\Po,k)t^k = \frac{ h^*_0 + h^*_1 t + \cdots + h^*_{d-1}t^{d-1} + h^*_d t^d }{(1-t)^{d+1}}.
\end{equation}
The numerator is
often called the $\mathit{h}^*$\emph{-polynomial} of $\Po$, and we
define the coefficients of the polynomial in the numerator of Lemma
\ref{lem:ratehr}, $h^*_0 + h^*_1 t + \cdots + h^*_{d-1}t^{d-1} + h^*_d
t^d$, as the $h^*$-\emph{vector} of $\Po$, which we write as $\ve
h^*(\Po) := (h^*_0, h^*_1, \dots, h^*_{d-1}, h^*_d)$.

A vector $(c_0,\ldots,c_d)$ is \emph{unimodal} if there exists an index~$p$, $0
\leq p \leq d$, such that $c_{i-1} \leq c_{i}$ for $i \leq p$ and $c_{j}
\geq c_{j+1}$ for $j \geq p$. Due to its algebraic implications,
several authors have studied the unimodality of $h^*$-vectors (see
\cite{Hibi1992Algebraic-Combi} and \cite{Stanley1996Combinatorics-a}
and references therein). 
It is well-known that if the Ehrhart ring of an integral polytope
$\Po$, $A(\Po)$, is Gorenstein, then $\ve h^*(\Po)$ is unimodal, and
symmetric \cite{Hibi1992Algebraic-Combi,Stanley1996Combinatorics-a}.
Nevertheless, the vector $\ve h^*(\Po)$ can be unimodal even when the
Ehrhart ring $A(\Po)$ is not Gorenstein. For matroid polytopes, their
Ehrhart ring is indeed often not Gorenstein.  For instance, De Negri and Hibi
\cite{Negri1997Gorenstein-Alge} prove explicitly when the
Ehrhart ring of a uniform matroid polytope is Gorenstein or not.  Two
fascinating facts, uncovered through experimentation, are that all
$h^*$-vectors seen thus far are unimodal, even for the cases when
their Ehrhart rings are not Gorenstein. In addition, when we computed
the explicit Ehrhart polynomials of matroid polytopes we observe their
coefficients are always positive. We conjecture:

\begin{conjecture} \label{unimodalconj} Let $\Po(M)$ be the matroid polytope of a matroid $M$.
    \begin{itemize}
        \item[(A)]  The $h^*$-vector of $\Po(M)$ is unimodal.
        \item[(B)]  The coefficients of the Ehrhart polynomial of $\Po(M)$ are positive. 
    \end{itemize}
\end{conjecture}

We have proved both parts of this conjecture in many instances. For
example, using computers, we were able to verify Conjecture
\ref{unimodalconj} for all uniform matroids up to $75$ elements as
well as for a wide variety of non-uniform matroids which are collected at
\cite{HawsMatroid-Polytop}.  We include here this information just
for the $28$ famous matroids presented in
\cite{Oxley1992Matroid-Theory}. Results in \cite{Katzman:math0408038},
with some additional careful calculations, imply that Conjecture
\ref{unimodalconj} is true for all rank $2$ uniform matroids.
Regarding part (A) of the conjecture we were also able to prove
\emph{partial unimodality} for uniform matroids of rank $3$, meaning
that the vector is non-decreasing up to a non-negative index, for
large enough $n$. Concretely we obtain:

\begin{theorem} \label{partialuni} \mbox{}
\begin{itemize}
\item[(1)] Conjecture \ref{unimodalconj} is true for all uniform
matroids up to $75$ elements and all uniform matroids of rank 2.
It is also true for all matroids listed in \cite{HawsMatroid-Polytop}.

\item[(2)] Let $\Po(U^{3,n})$ be the matroid polytope of a uniform
            matroid of rank $3$ on $n$ elements, and let $I$ be a non-negative
            integer.  Then there exists $n(I) \in \mathbb{N}$ such that for all
            $n \geq n(I)$ the $h^*$-vector of $\Po(U^{3,n})$, $\left(
                    h^*_0,\ldots, h^*_{n} \right)$, is non-decreasing
            from index 0 to $I$. That is, $ h^*_{0} \leq h^*_{1} \leq \dots
            \leq h^*_{I}$. 
\end{itemize}
\end{theorem}

\section{Computing the Ehrhart Polynomials} \label{sec:compehrhart}

Generating functions are crucial to proving our main results.  Let
$\Po \subseteq \R^n$ be a rational polyhedron. The \emph{multivariate
  generating function} of $\Po$ is defined as the formal Laurent series in $\Z[[z_1,\ldots,z_n,z_1^{-1},\ldots,z_n^{-1}]]$
\begin{equation*}
\tilde g_\Po ( \mathbf{z} ) = \sum_{\ve \alpha \in \Po \cap \Z^n} \mathbf{z}^\alpha ,
\end{equation*}
where we use the multi-exponent notation $\ve z^\vealpha =
\prod_{i=1}^n z_i^{\alpha_i}$.  If $\Po$ is bounded, $\tilde g_\Po$ is a
Laurent polynomial, which we consider as a rational function~$g_\Po$.
If $\Po$ is not bounded but is pointed (i.e., $\Po$ does not contain a
straight line), there is a non-empty open subset $U\subseteq\C^n$ such
that the series converges absolutely and uniformly on every compact
subset of~$U$ to a rational function~$g_\Po$ (see \cite{barvinok:99} and
references therein).  If $\Po$ contains a straight line, we set $g_\Po =
0$.  The rational function $g_\Po\in\Q(z_1,\dots,z_n)$ defined in this
way is called the \emph{multivariate rational generating function}
of~$\Po\cap\Z^n$.  Barvinok \cite{bar} proved that in polynomial time, when the
dimension of a polytope is fixed, $g_\Po$ can be represented
as a short sum of rational functions 
\begin{equation*}
 g_\Po(\ve z) =  \sum_{i \in I} \epsilon_i \frac{\ve z^{\ve a_{i}}}{\prod_{j=1}^{n} (1 - \ve z^{\ve b_{ij}} )},
\end{equation*}
where $\epsilon_i \in \{-1,1\}$.
Our first contribution is to show that in
the case of matroid polytopes of fixed rank, this still holds even
when their dimension grows.


\subsection{On the Tangent Cones of Matroid Polytopes} \label{subsec:ontangent}

Our goal is to compute the multivariate generating function of matroid
polytopes and independence matroid polytopes with fixed rank (later,
in \autoref{subsec:polymatroids}, we
will deal with the case of polymatroids), and to do this we will use a
crucial property of adjacent vertices.  Let $\ve v$ be a vertex of
$\Po$. Define the \emph{tangent cone} or \emph{supporting cone} of
$\ve v$ to be
\begin{equation*}
    \Co_\Po(\ve v) := \left\{\, \ve v + \ve w \; | \; \ve v + \epsilon \ve w \in \Po \text{ for some } \epsilon > 0 \, \right\}.
\end{equation*}

To illustrate our techniques we will use a running example throughout this section.
\begin{example}[Matroid on $K_4$] \label{ex:k4}
    Let $K_4$ be the complete graph on $4$ vertices. Label the ${4 \choose 2} = 6$ edges with $\{1,\ldots,6\}$. Every graph induces a matroid on its edges where the bases are all spanning trees (spanning forests for disconnected graphs) \cite{Welsh1976Matroid-Theory}. Let $M(K_4)$ be the matroid on the elements $\{1,\ldots,6\}$ with bases as all spanning trees of $K_4$. The rank of $M(K_4)$ is the size of any spanning tree of $K_4$, thus the rank of $M(K_4)$ is $3$. The $16$ bases of $M(K_4)$ are: 
$\{3, 5, 6\}$, $\{3, 4, 6\}$, $\{3, 4, 5\}$, $\{2, 5, 6\}$, $\{2, 4, 6\}$, $\{2, 4, 5\}$, $\{2, 3, 5\}$, $\{2, 3, 4\}$, $\{1, 5, 6\}$, $\{1, 4, 6\}$, $\{1, 4, 5\}$, $\{1, 3, 6\}$, $\{1, 3, 4\}$, $\{1, 2, 6\}$, $\{1, 2, 5\}$, $\{1, 2, 3\}$.
\end{example}

\begin{lemma}[See Theorem 4.1 in \cite{Gelfand1987Combinatorial-g}, Theorem 5.1 and Corollary 5.5 in \cite{Topkis1984Adjacency-on-Po}] \label{lem:adj}
    Let $M$ be a matroid. 
    \begin{itemize}
    \item[A)]   Two vertices $\ve e_{B_1}$ and $\ve e_{B_2}$ are adjacent in $\Po(M)$ if and only if $\ve e_{B_1} - \ve e_{B_2} = \ve e_i - \ve e_j$ for some $i,j$.
    \item[B)]   If two vertices $\ve e_{I_1}$ and $\ve e_{I_2}$ are adjacent
      in $\Po^\In (M)$ then $\ve e_{I_1} - \ve e_{I_2} \in \{ \, \ve e_i - \ve
      e_j, \, \ve e_i, \, -\ve e_j \, \}$ for some $i,j$. Moreover if $\ve v$
      is a vertex of $\Po^\In (M)$ then all adjacent vertices of $\ve v$ can
      be computed in polynomial time in $n$, even if the matroid $M$ is only
      presented by an evaluation oracle of its rank function~$\varphi$.
    \end{itemize}
\end{lemma}


Let $M$ be a matroid on $n$ elements with fixed rank $r$. Then the number of
vertices of $\Po(M)$ is polynomial in $n$. We can see this since the number of
vertices is equal to the number of bases of $M$, and the number of bases is
bounded by $n \choose r$, a polynomial of degree $r$ in $n$. Clearly the number of
vertices of $\Po^\In (M)$ is also polynomial in~$n$.  It is also clear that, 
when the rank~$r$ is fixed, 
all vertices of either polytope can be enumerated in polynomial time in~$n$,
even when the matroid is only presented by an evaluation oracle for its rank
function~$\varphi$.

Throughout this section we shall discuss polyhedral cones $\Co$ with
extremal rays $\{\ve r_1,\ldots,\ve r_l\}$ such that 
\begin{equation*}
    \ve r_k \; \in \; R_A := \{ \, \ve e_i - \ve e_j, \, \ve e_i, \, -\ve e_j \; | \;  i \in [n], \; j \in A \, \} \quad \text{for } k=1,\ldots,l
\end{equation*}
for some $A \subseteq [n]$. We will refer to $R_A$ as the
\emph{elementary set of $A$}. Note that by \autoref{lem:adj}
the rays of a tangent cone at a vertex $\ve e_A$ (corresponding to a set
$A\subseteq[n]$) of a matroid polytope or an independence matroid polytope form an
elementary set of $A$. Due to convexity and the
assumption that $\ve r_k$ are extremal, for each $i \in [n]$ and $j \in A$ at
most two of the three vectors $\ve e_i - \ve e_j, \ve e_i, - \ve e_j$ are
extremal rays $\ve r_k$ of $\Co$.  This implies by construction, that
    considering all pairs $\ve e_i - \ve e_j$ and $\ve e_i$ or $-\ve
    e_j$, the number of generators $\ve r_k$ of $\Co$ is bounded by
\begin{equation} \label{eq:genbnd}
n|A| + n + |A|.
\end{equation}

Recall a cone is \emph{simple} if it is generated by linearly independent vectors and
it is \emph{unimodular} if its fundamental parallelepiped contains only $\ve 0$
from $\Z^n$ \cite{1997Handbook-of-dis}.  A triangulation of $\Co$ is
\emph{unimodular} if it is a polyhedral subdivision such that each sub-cone is
unimodular.

\begin{figure}[!htb]
    \centerline{
    \ifpdf
    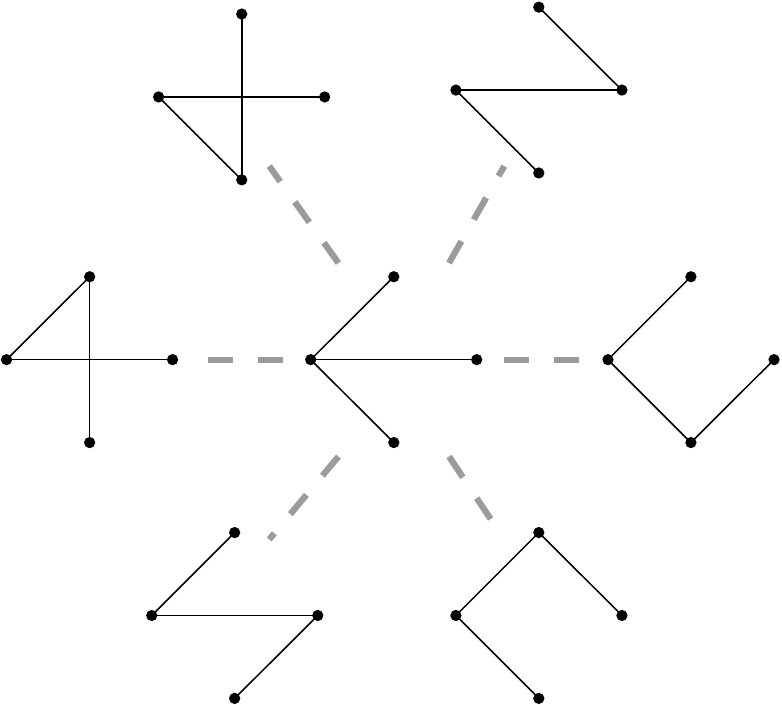
    \else
    \input{K4adj.pstex_t}
    \fi
    }
    \caption{${\{2, 3, 5\}}$, ${\{2, 3, 4\}}$, ${\{1, 3, 6\}}$, ${\{1, 3, 4\}}$, ${\{1, 2, 6\}}$ and ${\{1, 2, 5\}}$ are spanning trees of $K_4$ that differ from ${\{1, 2, 3\}}$ by adding one edge and removing one edge.}\label{fig:K4adj}
\end{figure}
\begin{example}[Matroid on $K_4$]
    The vertices $\ve e_{\{2, 3, 5\}}$, $\ve e_{\{2, 3, 4\}}$, $\ve e_{\{1, 3, 6\}}$, $\ve e_{\{1, 3, 4\}}$, $\ve e_{\{1, 2, 6\}}$ and $\ve e_{\{1, 2, 5\}}$ are all adjacent to the vertex $\ve e_{\{1, 2, 3\}}$, see \autoref{fig:K4adj}. Moreover, the tangent cone $\Co_{\Po(M(\ve K_4))}(\ve e_{\{1, 2, 3\}})$ is generated by the differences of these vertices with $\ve e_{\{1, 2, 3\}}$:
\begin{multline*}
    \Co_{\Po(M(\ve K_4))}(\ve e_{\{1, 2, 3\}}) = \ve e_{\{1, 2, 3\}} + \cone \left\{\ve e_{\{2, 3, 5\}}-\ve e_{\{1, 2, 3\}}, \ve e_{\{2, 3, 4\}}-\ve e_{\{1, 2, 3\}}, \right. \\
           \left. \ve e_{\{1, 3, 4\}}-\ve e_{\{1, 2, 3\}},\ve e_{\{1, 2, 6\}}-\ve e_{\{1, 2, 3\}},\ve e_{\{1, 2, 5\}}-\ve e_{\{1, 2, 3\}} \right\}
\end{multline*}
\end{example}

\begin{lemma} \label{lem:eluni}
    Let $\Co \subseteq \R^n$ be a cone generated by $p$ extremal rays $\{\ve
    r_1,\ldots,\ve r_p\} \subseteq R_A$ where $R_A$ is an elementary set of
    some $A \subseteq [n]$. Every triangulation of $\Co$ is unimodular.
\end{lemma}

\begin{proof}
Without loss of generality, we can assume $\{\ve r_1,\ldots,\ve r_l\}$
are generators of the form $\ve e_i - \ve e_j$ and $\{\ve r_{l+1},\ldots,\ve r_p\}$ are
generators of the form $\ve e_i$ or $-\ve e_j$ for the cone $\Co$. 

It is easy to see that the matrix $\tilde T_\Co := [\ve r_1,\ldots,\ve r_l]$ is totally
unimodular. Let $G_\Co$ be a directed graph with vertex set $[n]$ and
an edge from vertex $i$ to $j$ if $\ve r_k = \ve e_i - \ve e_j$ is an extremal ray of
$\Co$. We can see that $G_\Co$ is a subgraph of the complete directed graph $K_n$ with two arcs between each pair of vertices; one for each direction.
Since $\tilde T_\Co := [\ve r_1,\ldots,\ve r_l]$ is the incidence matrix of the graph $G_\Co$,
      it is totally unimodular
      \cite[Ch.~19,~Ex.~2]{Schrijver1986Theory-of-linea}, i.e., every subdeterminant is $0$, $1$ or $-1$
      \cite[Ch.~19, Thm.~9]{Schrijver1986Theory-of-linea}.  Therefore $T_\Co :=
      [\ve r_1,\ldots,\ve r_l,\ve r_{l+1},\ldots,\ve r_p]$ is totally unimodular since
      augmenting $\tilde T_\Co$ by a vector $\ve e_i$ or $-\ve e_j$ preserves this
      subdeterminant property: for any submatrix containing part of a vector
      $\ve e_i$ or $-\ve e_j$ perform the cofactor expansion down the vector $\ve e_i$ or
      $-\ve e_j$ when calculating the determinant.  

Since $T_\Co$ is totally unimodular, each basis of $T_\Co$ generates the entire integer
lattice $\Z^n \cap \lin(\Co)$ and hence every simplicial cone of a triangulation has
normalized volume $1$.
\end{proof}

\begin{lemma} \label{lem:aug}
    Let $\Co \subseteq \R^n$ be a cone generated by $l$ extremal
    rays $\{\ve r_1,\ldots,\ve r_l\} \subseteq R_A$ where $R_A$ is an elementary set of some $A \subseteq [n]$, where $\dim(\Co) < n$. The extremal rays
    $\{\ve r_1,\ldots,\ve r_l\}$ can be augmented by a vector $\ve{\tilde r}$ such that
    $\dim ( \cone \{ \ve r_1,\ldots,\ve r_l,\ve{\tilde r}\}) = \dim(\Co) +1$,
    the vectors
    $\ve r_1,\ldots,\ve r_l,\ve{\tilde r}$ are all extremal,
    and $\ve{\tilde r} \in R_A$.
\end{lemma}
\begin{proof}
    It follows from convexity that at most two of $\ve e_i-\ve e_j$, $\ve e_i$ or
    $-\ve e_j$ are extremal generators of $\Co$ for $i \in [n]$ and $j \in A$.
    There are at least $n$ possible extremal ray generators, considering two of
    $\ve e_i-\ve e_j$, $\ve e_i$ or $-\ve e_j$ for each $i \in [n]$ and $j \in A$.
    Moreover, all these pairs span $\R^n$. Thus by the basis augmentation theorem
    of linear algebra, there exists a vector $\ve{\tilde r}$ such that $\dim (
            \cone \{ \ve r_1,\ldots,\ve r_l,\ve{\tilde r}\} )  = \dim(\Co) +1$ and
    $\ve r_1,\ldots,\ve r_l,\ve{\tilde r}$ are all extremal. 
\end{proof}

\begin{lemma} \label{lem:elsimpbnd}
    Let $r$ be a fixed integer, $n$ be an integer, $A \subseteq [n]$ with $|A| \leq r$ and
    let $\Co \subseteq \R^n$ be a cone generated by $l$ extremal rays $\{\ve
    r_1,\ldots,\ve r_l\} \subseteq R_A$ where $R_A$ is an elementary set of $A$. Then any triangulation of $\conv(\{\ve 0, \ve
            r_1,\ldots,\ve r_l\})$ has at most a polynomial in $n$ number of
    top-dimensional simplices.
\end{lemma}
\begin{proof}

Assume $\dim(\Co) = n$. Later, we will show how to remove this restriction.
We can see that
$    \conv \{\ve  0, \ve r_1,\ldots,\ve r_l \} \; \subseteq \;  [-1,1]^{A} \times \tilde \Delta_{[n] \setminus A}$
where
\begin{align*}
    [-1,1]^{A} & :=  \{ \, \ve x \in \R^A \mid |\ve x_j| \leq 1 \ \ j \in A \, \}  \subseteq \R^A  \\
    \tilde \Delta_{[n] \setminus A} & := \conv ( \{ \, \ve e_i \; | \; i \in [n] \setminus A \, \} \cup \{\ve 0\} )  \subseteq \R^{[n] \setminus A}.
\end{align*}
The volume of a $d$-simplex $\conv\{\ve v_0,\ldots,\ve v_d\}$ is \cite{1997Handbook-of-dis}
\begin{equation}
    \label{eq:simpvol}
    \frac{1}{d!} \left| \det \left( \begin{array}{ccc} \ve v_0 & \cdots & \ve v_d \\ \ve 1 & \cdots & \ve 1 \end{array} \right) \right| .
\end{equation}
Thus the $(n-|A|)$-volume of $\tilde \Delta_{[n] \setminus A}$ is $ 
\frac{1}{(n-|A|)!}$ and the $|A|$-volume of $[-1,1]^{A}$ is $2^{|A|}$.
Therefore
\begin{equation*}
    \vol \big( [-1,1]^{A} \times \tilde \Delta_{[n] \setminus A} \big) =
    2^{|A|} \frac{1}{(n-|A|)!} = \frac1{n!} 2^{|A|} n(n-1) \cdots (n-|A| +1) .
\end{equation*}
It is also a fact that any integral $n$-simplex has $n$-volume bounded below by
$\frac{1}{n!}$, using the simplex volume equation \eqref{eq:simpvol}. Therefore
any triangulation of $\conv \{ \ve 0,\ve r_1,\ldots,\ve r_l \}$ has at most 
\begin{equation*}
  2^{|A|} n(n-1) \cdots (n-|A| +1) \leq
  2^{r} n(n-1) \cdots (n-r +1)
\end{equation*}
full-dimensional simplices, a polynomial function in~$n$ of degree~$r$.

Let $d_\Co := n-\dim( \Co )$. If $\dim(\Co) < n$, then by Lemma \ref{lem:aug}, $\{\ve r_1,\ldots,\ve r_l\}$ can
be augmented with vectors $\{\ve{\tilde r}_1, \ldots, \ve{\tilde r}_{d_\Co} \}$ where $ \ve{\tilde
r}_k \in R_A$ for $A$ above, such that $\dim( \cone \{ \ve r_1,\ldots,\ve r_l,\allowbreak \ve{\tilde r}_1, \ldots, \ve{\tilde r}_{d_\Co} \} ) = n$ and $\{
\ve r_1,\ldots,\ve r_l,\ve{\tilde r}_1, \ldots, \ve{\tilde r}_{d_\Co} \}$ are extremal. Moreover,
\begin{equation*}
    \begin{array}{c}
    \dim ( \conv \{\ve 0,\ve r_1,\ldots,\ve r_l\} ) < \dim ( \conv \{\ve 0, \ve r_1,\ldots,\ve r_l,\ve{\tilde r}_1\} ) < \cdots  \\
    < \dim ( \conv \{\ve 0 ,\ve  r_1,\ldots,\ve  r_l,\ve{\tilde r}_1,\ldots,
    \ve{\tilde r}_{d_\Co -1}\} ) 
    < \dim ( \conv \{\ve 0,\ve r_1,\ldots,\ve r_l,\ve{\tilde r}_1,\ldots, \ve{\tilde r}_{d_\Co}\} ),
    \end{array}
\end{equation*}
that is, $\ve{\tilde r}_k \notin \affine \{\ve 0, \ve r_1,\ldots,\ve r_l, \ve{\tilde r}_1, \ldots, \ve{\tilde r}_{k-1}\}$ for $1 \leq k \leq d_\Co$. 

\begin{figure}[!htb]
    \centerline{ \ifpdf
    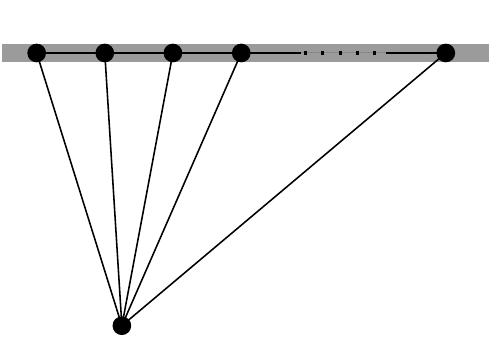
    \else
    \input{lemnsimp.pstex_t}
    \fi }
    \caption{If $\ve{\tilde r}_k$ is not contained in the affine span of
      $\{\ve 0, \ve r_1,\ldots,\ve r_p, \ve{\tilde r}_1,\ldots, \ve{\tilde
        r}_{k-1}\}$ then every full-dimensional simplex must contain
      $\ve{\tilde r}_k$.}
    \label{fig:lemnsimp}
\end{figure}

Since $\ve{\tilde r}_k \notin \affine \{ \ve 0, \ve r_1,\ldots,\ve r_l,
        \ve{\tilde r}_1, \ldots, \ve{\tilde r}_{k-1} \} $, any full-dimensional
simplex in a triangulation of $\conv \{ \ve 0, \ve r_1,\ldots,\ve r_l,
        \ve{\tilde r}_1, \ldots, \ve{\tilde r}_k \} $ must contain $\ve{\tilde
    r}_k$, see Figure \ref{fig:lemnsimp}. If not, then there exists a
    top-dimensional simplex using the points $\{ \ve 0, \ve r_1,\ldots,\ve
    r_l, \allowbreak \ve{\tilde r}_1, \ldots, \ve{\tilde r}_{k-1} \}$, but we know all
    these points lie in a subspace of one less dimension, a contradiction. Therefore, a bound on
    the number of simplices in a triangulation of $\conv \{ \ve 0, \ve
            r_1,\ldots,\ve r_l, \ve{\tilde r}_1, \ldots, \ve{\tilde r}_k \}$ is
    a bound on that of $\conv \{ \ve 0, \ve r_1,\ldots,\ve r_l,\allowbreak \ve{\tilde r}_1,
            \ldots, \ve{\tilde r}_{k-1} \}$.  

Thus, if $\dim (\Co) < n$ we can augment $\Co$ by vectors $\ve{\tilde
r}_1,\ldots, \ve{\tilde r}_{d_\Co}$ so that the cone $\tilde \Co :=
\cone \{ \ve r_1,\ldots,\ve r_p,\ve{\tilde r}_1,\ldots, \ve{\tilde
r}_{d_\Co} \}$ is of dimension $n$ and $ \ve r_l, \ve{\tilde r}_k \;
\in \; R_A$ for $A$ above. We proved any triangulation of $\conv \{
\ve 0, \ve r_1,\ldots,\ve r_p,\ve{\tilde r}_1,\ldots, \ve{\tilde r}_{d_\Co} \}
$ has at most polynomially many full-dimensional $n$-simplices, which
implies that any triangulation of $\conv \{ \ve 0, \ve r_1,\ldots,\ve
r_p \}$ has at most polynomially many top-dimensional
simplices due to the construction of the generators
$\ve{\tilde r}_k$.
\end{proof}

We have shown that for a cone $\Co$ generated by an elementary set of
extremal rays $\{\ve r_1,\ldots,\ve r_l\} \subseteq R_A$ for some $A
\subseteq [n]$, any triangulation of $\conv\{\ve 0,\ve r_1,\ldots,\ve
r_l\}$ has at most polynomially many simplices. What we need next is
an efficient method to compute some triangulation of $\conv\{\ve
0,\ve r_1,\ldots,\ve r_l\}$. We will show that the placing
triangulation is a suitable candidate.

Let $\Po \subseteq \R^n$ be a polytope of dimension $n$ and $\Delta$
be a facet of $\Po$ and $\ve v \in \R^n$. There exists a unique
hyperplane $H$ containing $\Delta$ and $\Po$ is contained in one of
the closed sides of $H$, call it $H^+$. If $\ve v$ is contained in the
interior of $H^-$, the other closed halfspace defined by $H$, then
$\Delta$ is \emph{visible} from $\ve v$ (see chapter 14.2 in
\cite{1997Handbook-of-dis}). The well-known \emph{placing
triangulation} is given by an algorithm where a point is added to an
intermediate triangulation by determining which facets are visible to
the new point
\cite{1997Handbook-of-dis, De-Loera2006Triangulations}. We recall now how
to determine if a facet is visible to a vertex in polynomial time.

\begin{figure}[!htb]    
    \centering
    \hfil(a)\subfigure{\scalebox{0.8}{\ifpdf
    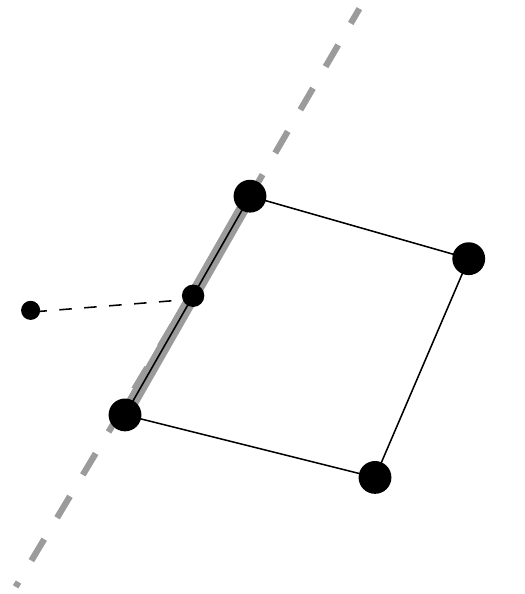
    \else
    \input{visfacet1.pstex_t}
    \fi}}
    \hfil(b)\subfigure{\scalebox{0.8}{\ifpdf
    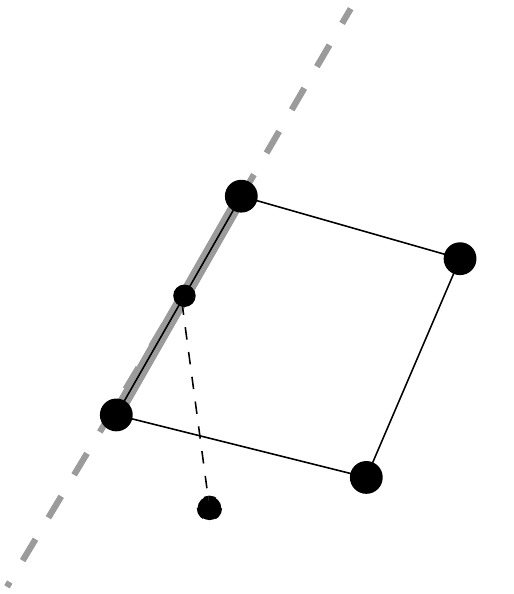
    \else
    \input{visfacet2.pstex_t}
    \fi}}
    \caption{Figure (a) shows $\Delta$ visible to $\ve v$. Figure (b) shows
        $\Delta$ {\bf not} visible to $\ve v$.}
\end{figure}

\begin{lemma} \label{lem:vislp}
Let $\Po \subseteq \R^n$ be a polytope given by $t$ vertices
$\{\ve v^1,\ldots,\ve v^t\} \subseteq \R^n$ and $\Delta \subseteq \Po$ be a facet of $\Po$
given by $q$ vertices $\{\ve{\tilde v}^1,\ldots, \ve{\tilde v}^q \} \subseteq
\{\ve v^1,\ldots,\ve v^t\} $. If $\ve v \in \R^n$ where $\ve v \notin \Po$ then deciding if
$\Delta$ is visible to $\ve v$ can be done in polynomial time in the input
$\{\ve{\tilde v}^1,\ldots, \ve{\tilde v}^q \}$, $\{\ve v^1,\ldots,\ve v^t\}$ and $\ve v$.
\end{lemma}

\begin{proof}
Let $  \ve  z := \frac{1}{q}
\sum_{i=1}^q  \ve{\tilde v}^i $ so that $\ve z \in \interior ( \Delta )$. We consider the linear program:
\begin{equation} \label{vislp}
    \begin{split} 
    \Bigg\{ \, \left( \begin{array}{c} \ve x \\ \ve y \\ \lambda \end{array} \right) \in \R^{n+t+1} \;\; | \;\;  \ve x = \sum_{i=1}^t \ve v^i \ve y_i, \: & \ve y \geq \ve 0, \: \sum_{i=1}^t \ve y_i=1,    \\
    & \: 0 \leq \lambda < 1, \: \lambda \ve v + (1-\lambda) \ve z = \ve x \, \Bigg\} .  
    \end{split}
\end{equation}

If \eqref{vislp} has a solution then there exists a point $\ve{\bar x} \in \Po$
between the facet $ \Delta$ and $\ve v$, hence $ \Delta$ is not visible from $\ve v$. If
\eqref{vislp} does not have a solution, then there are no points of $
\Po$ between $\ve v$ and $\Delta$, hence $\Delta$ is visible from $\ve v$
(see Lemma 4.2.1 in \cite{De-Loera2006Triangulations}). It is well known that
a strict inequality, such as the one in \eqref{vislp}, can be handled by an
equivalent linear program which has only one additional variable. Determining
if \eqref{vislp} has a solution can be done in polynomial time in the input
\cite{Schrijver1986Theory-of-linea}.
\end{proof}

\begin{algorithm}[The Placing Triangulation \cite{1997Handbook-of-dis, De-Loera2006Triangulations}]
    \mbox{}
    \vskip .25cm
    \label{alg:ept}
    \begin{center}
        \begin{tabular}{l}
            \begin{minipage}{.85\linewidth}
                \begin{algorithmic}
                    \item[Input:] A set of ordered points $\{ \ve v_1,\ldots,\ve v_t \} \in \R^n$.
                    \item[Output:] A triangulation $\Tr$ of $\{ \ve v_1,\ldots,\ve v_t \}$
                    \STATE $\Tr := \{\{\ve v_1\}\}$.
                    \FOR{ each $\ve v_i \in \{\ve v_2,\ldots,\ve v_t\}$} \label{alg:ept:l4}
                        \STATE Let $B \in \Tr$.
                        \STATE $P_i := \{\ve v_1,\ldots,\ve v_{i-1}\}$
                        \IF{ $\ve v_i \notin \affine(P_i)$} \label{alg:ept:aff} 
                            \STATE $\Tr' := \emptyset$
                            \FOR{ each $D \in \Tr$} \label{alg:ept:l9}
                                \STATE $\Tr' := \Tr' \cup \{D \cup \{\ve v_i\}\}$.
                            \ENDFOR
                        \ELSE
                            \FOR{ each $B \in \Tr$ and each $(|B|-1)$-subset $C$ of $B$} \label{alg:ept:l12}
                                \STATE Create and solve the linear program \eqref{vislp} with $(P_i,C,\ve v_i)$ to decide visibility of $C$ to $\ve v_i$. \label{alg:ept:l13}
                                \IF{$C$ is visible to $\ve v_i$}
                                    \STATE $\Tr' := \Tr' \cup \{C \cup \{\ve v_i\}\}$
                                \ENDIF
                            \ENDFOR
                        \ENDIF
                        \STATE $\Tr := \Tr'$
                    \ENDFOR
                    \RETURN $\Tr$
                \end{algorithmic}
            \end{minipage}\\
        \end{tabular}
    \end{center}
\end{algorithm}

Indeed, Algorithm \ref{alg:ept} returns a triangulation
\cite{1997Handbook-of-dis}. We will show that for certain input, it runs in
polynomial time. We remark that there are exponentially, in $n$, many lower
dimensional simplices in any given triangulation. But, it is important to note
that only the highest dimensional simplices are listed in an intermediate
triangulation (and thus the final triangulation) in the placing triangulation
algorithm.

\begin{theorem} \label{thm:ept}
Let $r$ be a fixed integer, $n$ be an integer, $A \subseteq [n]$ with $|A|
\leq r$, and let $\Co \subseteq \R^n$ be a cone generated by extremal rays
$\{\ve r_1,\ldots,\ve r_l\} \subseteq R_A$. Then the placing triangulation
(Algorithm \ref{alg:ept}) with input $\{\ve 0, \ve r_1,\ldots,\ve r_l\}$ runs
in polynomial time. 
\end{theorem}
\begin{proof}
By Equation \eqref{eq:genbnd} there is only a polynomial, in $n$, number of
extremal rays $\{\ve r_1,\ldots,\ve r_l\}$. Thus, the {\bf for}
statement on line \ref{alg:ept:l4} repeats a polynomial number of
times. Step \ref{alg:ept:aff} can be done in polynomial time by solving the linear
equation $[\ve r_1,\ldots,\ve r_{i-1}] \ve x = \ve v_i$.  

The {\bf for} statement on line \ref{alg:ept:l9} repeats for every simplex $D$ in the
triangulation $\Tr$, and the number of simplices in $\Tr$ is bounded by the
number of simplices in the final triangulation. By Lemma \ref{lem:elsimpbnd}
any triangulation of extremal cone generators in $R_A$ with the origin will use
at most polynomially many top-dimensional simplices. Hence the number of top-dimensional simplices of any
partial triangulation $\Tr$ will be polynomially bounded since it is a subset
of the final triangulation.

The {\bf for} statement on line \ref{alg:ept:l12} repeats for every simplex $B$ and
every $(|B|-1)$-simplex of $B$. As before, the number of simplices $B$ is
polynomially bounded, and there are at most $n$ $(|B|-1)$-simplices of $B$.
Thus the {\bf for} statement will repeat a polynomial number of times.

Finally, by Lemma \ref{lem:vislp}, determining if $C$ is visible to $\ve v_i$ can
be done in polynomial time. Therefore Algorithm \ref{alg:ept} runs in a
polynomial time.
\end{proof}

\begin{corollary} \label{elemconepoly}
Let $r$ be a fixed integer, $n$ be an integer, $A \subseteq [n]$ with $|A|
\leq r$, and let $\Co \subseteq \R^n$ be a cone generated by extremal rays
$\{\ve r_1,\ldots,\ve r_l\} \subseteq R_A$.  A triangulation of $\Co$ can be
computed in polynomial time in the input of the extremal ray generators $\{\ve
r_1,\ldots,\ve r_l \}$. 
\end{corollary}

\begin{proof}
Let $\Po_\Co := \conv \{ \ve 0, \ve r_1,\ldots,\ve r_t\} $. We give an algorithm which produces a triangulation of
    $\Po_\Co := \conv \{ \ve 0, \ve r_1,\ldots,\ve r_t\} $ such that each full-dimensional
    simplex has $\ve 0$ as a vertex. Such a triangulation would extend to a triangulation of
    the cone $\Co$. This can be accomplished by applying two placing
    triangulations: one to triangulate the boundary of $\Po_\Co$ not incident to
    $\ve 0$, and another to attach the triangulated boundary faces to $\ve 0$.  The
    algorithm goes as follows:

\begin{itemize}
    \item[1)] {\bf Triangulate $\Po_\Co$ using the placing triangulation algorithm. Call it $\Tr'$.} 

    \item[2)] {\bf Triangulate $\Po_\Co$ using the boundary faces of $\Tr'$ which do not contain $v$}. 
\begin{algorithm}[Triangulation joining $\ve 0$ to boundary faces]
    \mbox{}
    \vskip .25cm
    \label{alg:jt}
    \begin{center}
        \begin{tabular}{l}
            \begin{minipage}{.85\linewidth}
                \begin{algorithmic}
                        \item[Input:] A triangulation $\Tr'$ of $\Po_\Co$, given by its vertices.
                        \item[Output:] A triangulation $\Tr$ of $\Po_\Co$ such that every highest dimension simplex of $\Tr$ is incident to $\ve 0$.
                        \STATE $\Tr := \emptyset$ \label{alg:jt:1}
                        \FOR{each $C$ where $C$ is a $(|B|-1)$-simplex of $B \in \Tr'$} \label{alg:jt:3}
                            \IF{$C$ is not a $(|A|-1)$-simplex of $A \in \Tr'$ where $A \neq B$} \label{alg:jt:4}
                                \STATE $\Tr := \Tr \cup \{ \, C \cup \{\ve 0\} \, \}$ \label{alg:jt:5}
                            \ENDIF
                        \ENDFOR
                        \RETURN $\Tr$

                \end{algorithmic}
            \end{minipage}\\
        \end{tabular}
    \end{center}
\end{algorithm}
\end{itemize}

\begin{figure}[!htb]    
    \centering
    \hfil(a)\subfigure{\scalebox{0.6}{\ifpdf
    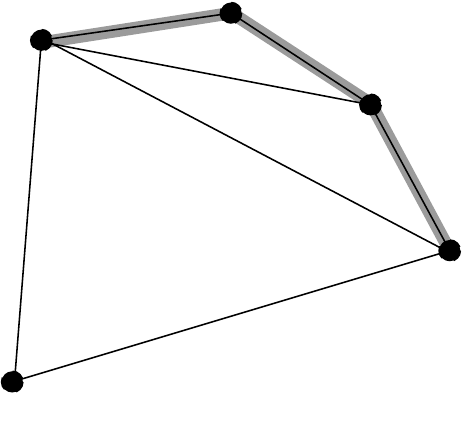
    \else
    \input{bndtri1.pstex_t}
    \fi}}
    \hfil(b)\subfigure{\scalebox{0.6}{\ifpdf
    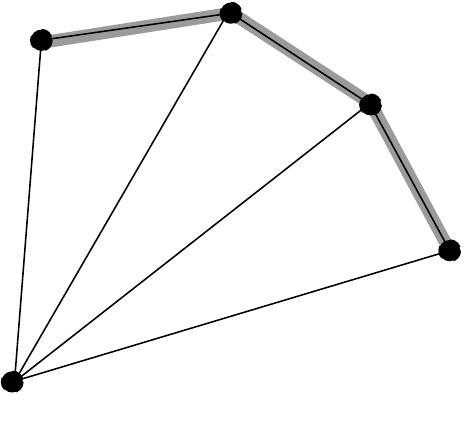
    \else
    \input{bndtri2.pstex_t}
    \fi}}
    \caption{A triangulation $\Tr'$ of $\Po_\Co$ can be used to extend to a triangulation $\Tr$ such that $\ve 0$ is incident to every highest dimensional simplex}
\end{figure}

By Theorem \ref{thm:ept}, triangulating $\Po_\Co$ using Algorithm \ref{alg:ept}
can be done in polynomial time. Algorithm \ref{alg:jt} indeed produces a triangulation of $\Po_\Co$. It covers
$\Po_\Co$ since every extremal ray generator $\ve r_k$ of $\Co$ is on some
$(\dim(\Co)-1)$-simplex. Moreover, $\Tr$ by construction has the property that
the intersection of any two simplices of $\Tr$ is a simplex. Step
\ref{alg:jt:4} checks if $C$ is on the boundary, since if $C$ is on the
boundary it will not be on the intersection of two higher-dimensional
simplices. 

Step \ref{alg:jt:3} repeats a polynomial number of times since any
triangulation of $\Po_\Co$ has at most a polynomial number of simplices, and
each simplex $B$ has at most $n$ $(|B|-1)$-simplices. Step \ref{alg:jt:4} can
be computed in polynomial time since again there are only polynomially many
simplices $B$ in the triangulation $\Tr'$ and at most $n$ $(|B|-1)$-simplices
to check if they are equal to $C$. Hence, Algorithm \ref{alg:jt} runs in
polynomial time.
\end{proof}

\begin{example}
    The tangent cone at the vertex $\ve e_B := \ve e_{\{1,2,3\}}$ on the polytope $\Po(M(K_4))$ can be triangulated as:
\begin{align*}
\Big\{ \, & \left\{ \ve e_{\{2, 3, 5\}}- \ve e_B,\, \ve e_{\{2, 3, 4\}}- \ve e_B,\, \ve e_{\{1, 3, 6\}}- \ve e_B,\, \ve e_{\{1, 3, 4\}}- \ve e_B,\, \ve e_{\{1, 2, 6\}} - \ve e_B \right\},  \\
& \left\{ \ve e_{\{2, 3, 5\}}- \ve e_B,\, \ve e_{\{1, 3, 6\}}- \ve e_B,\, \ve e_{\{1, 3, 4\}}- \ve e_B,\, \ve e_{\{1, 2, 6\}}- \ve e_B,\, \ve e_{\{1, 2, 5\}}- \ve e_B \right\}, \\
& \left\{ \ve e_{\{2, 3, 5\}}- \ve e_B,\, \ve e_{\{2, 3, 4\}}- \ve e_B,\, \ve e_{\{1, 3, 4\}}- \ve e_B,\, \ve e_{\{1, 2, 6\}}- \ve e_B,\, \ve e_{\{1, 2, 5\}}- \ve e_B \right\} \, \Big\}.
\end{align*}
\end{example}

\subsection{Polymatroids}
\label{subsec:polymatroids}
We will show that certain lemmas from Subsection
    \ref{subsec:ontangent} also hold for certain polymatroids.
 Recall that the rank of the matroid $M$ is the size of any bases of $M$ which
 equals $\varphi([n])$. Our lemmas from Subsection \ref{subsec:ontangent} rely
 on the fact that $M$ has fixed rank, that is, for some $r \in \Z$, $r \geq 0$,
 $\varphi(A) \leq r$ for all $A \subseteq [n]$. We will show that a similar
 condition on a polymatroid rank function is sufficient for the lemmas of
 Subsection \ref{subsec:ontangent} to hold.  

 \begin{lemma} \label{lem:polyvert}
    Let $\psi \colon 2^{[n]} \longrightarrow \mathbb{N}$ be an integral polymatroid
    rank function where $\psi(A) \leq r$ for all $A \subseteq [n]$, where $r$ is
    a fixed integer. Then the number of vertices of $\Po(\psi)$ is bounded by a
    polynomial in $n$ of degree $r$. 
 \end{lemma}
\begin{proof}
  It is known that if $\psi$ is integral then all vertices of
    $\Po(\psi)$ are integral \cite{Welsh1976Matroid-Theory}. The number of vertices
    of $\Po(\psi)$ can be bounded by the number of non-negative integral solutions to $x_1 +
    \cdots +  x_n \leq r$, which has ${n + r \choose r}$
    solutions, a polynomial in $n$ of degree $r$ \cite{Stanley1997Enumerative-Com}.
\end{proof}

\begin{lemma} \label{lem:polyenum}
    Let $\psi \colon 2^{[n]} \longrightarrow \mathbb{N}$ be an integral polymatroid
    rank function. If $\ve v$ is a vertex of $\Po(\psi)$ then all adjacent
    vertices of $\ve v$ can be enumerated in polynomial time. Moreover if
    $\psi(A) \leq r$ for all $A \subseteq [n]$, where $r$ is a fixed integer,
    then the vertices of $\Po(\psi)$ can be enumerated in polynomial time.
\end{lemma}

\begin{proof}
     If $\ve v$ is a vertex of $\Po(\psi)$
     then generating and listing all adjacent vertices to $\ve v$ can
     be done in polynomial time by Corollary 5.5 in
     \cite{Topkis1984Adjacency-on-Po}. If $\psi(A) \leq r$ for all $A
     \subseteq [n]$, where $r$ is a fixed integer, then, by Lemma
     \ref{lem:polyvert}, there is a polynomial number of vertices for
     $\Po(\psi)$.  We know that $\ve 0 \in \R^n$ is a vertex of any
     polymatroid. Therefore, beginning with $\ve 0$, we can perform a
     breadth-first search, which is output-sensitive polynomial time, on the graph of
     $\Po(\psi)$, enumerating all vertices of $\Po(\psi)$.
\end{proof}

What remains to be shown is that these polymatroids have cones like the
ones in Subsection \ref{subsec:ontangent}.  We first recall some needed
definitions from \cite{Topkis1984Adjacency-on-Po}.  Let $\ve v,\ve w \in
\R^n$ and define $\Delta(\ve v,\ve w) := \{ \, i \in [n] \mid v_i \neq w_i
\, \}$ and $\cl(\ve v) := \{ \, S \mid S \subseteq [n], \ \sum_{i \in S}
v_i = \psi(S) \, \}$. Let $F = \{ f_1,\ldots,f_{|F|} \}$ be an ordered
subset of $[n]$ and $F_i := \{f_1,\ldots,f_i\}$. If $\psi$ is a polymatroid
rank function then we construct $\ve v \in \R^n$ where $v_i = \psi(F_i) -
\psi(F_{i-1})$ where $ v_j = 0$ when $j \notin F$ and one says $F$
\emph{generates} $\ve v$. A classical result of Edmonds \cite{Edmonds2003Submodular-func} says that the set of vectors generated by all ordered subsets of $[n]$
is exactly the set of vertices of $\Po(\psi)$. Now we can restate an important
lemma. 

\begin{lemma}[See Theorem 4.1 and Section 2 in \cite{Topkis1984Adjacency-on-Po}.] \label{lem:topadj}
    Let $\psi$ be a polymatroid rank function. If $\ve v$ and $\ve w$ are vertices
    of the polymatroid $\Po(\psi)$ then either
    \begin{itemize}
        \item[(i)] $|\Delta(\ve v, \ve w)| = 1$ or
        \item[(ii)] $\cl(\ve v) = \cl(\ve w)$ and $\Delta(\ve v,\ve w) =
        \{c,d\}$ for some $c,d \in [n]$ where there exists some ordered set $F
        = \{f_1,\ldots,f_{|F|}\}$ which generates $\ve v$ with $f_{k+1} = d$
        and $f_k = c$ for some integer $k$, $1 \leq k \leq |F| -1$; moreover
        the ordered set $\tilde F :=
        \{f_1,\ldots,f_{k-1},f_{k+1},f_k,f_{k+2},\ldots,f_{|F|}\}$ generates
        $\ve w$.
    \end{itemize}
\end{lemma}

\begin{lemma} \label{lem:polyelem}
    Let $\psi$ be an integral polymatroid rank function and $\Co$ the tangent
    cone of a vertex $\ve v$ of the polymatroid $\Po(\psi)$, translated to the
    origin. Then $\Co$ is generated by extremal ray generators $\{\ve
    r_1,\ldots, \ve r_l\} \subseteq R_{\supp(\ve v)}$, where $R_{\supp(\ve v)}$
    is an elementary set of $\supp(\ve v)$.
\end{lemma}
\begin{proof}
Let $\psi \colon 2^{[n]} \longrightarrow \Z$ be a integral polymatroid rank function.
    Let $\ve v$ and $\ve w$ be adjacent vertices of the polymatroid
    $\Po(\psi)$. Using Lemma \ref{lem:topadj}, if $|\Delta(\ve v, \ve w)| = 1$
    then $\ve w - \ve v = h \ve e_i$ where $h$ is some integer and $\ve e_i$ is the
    standard $i$th elementary vector for some $i \in [n]$. If $h < 0$ then certainly $i
    \in \supp(\ve v)$, else $i \in [n]$. Thus $\ve w - \ve v$, a generator of
    $\Co$, is parallel to a vector in $R_{\supp(\ve v)}$.

Let $\ve v$ and $\ve w$ be adjacent and satisfy ({\bf ii}) in Lemma
\ref{lem:topadj}, where $\Delta(\ve v,\ve w) = \{c,d\}$. Hence there exists an
$F = \{f_1,\ldots,f_{|F|}\}$ which generates $\ve v$ with $f_{k+1} = d$ and
$f_k = c$ for some integer $k$, $1 \leq k \leq |F| -1$; moreover the ordered
set $\tilde F := \{f_1,\ldots,f_{k-1},f_{k+1},f_k,f_{k+2},\ldots,f_{|F|}\}$
generates $\ve w$. First we note that $\psi(F_{k-1}) = \psi(\tilde F_{k-1})$
and $\psi(F_{k+1}) = \psi(\tilde F_{k+1})$. By assumption, we know $ v_{c}
\neq  w_{c}$, $v_{d} \neq w_{d}$ and $v_l = w_l$ for all $l
    \in [n] \setminus \{c,d\}$. Thus
\begin{align*}
    (\ve v - \ve w)_{c} & =  v_{c} -  w_{c} & = \phantom{-} & \psi(F_{k+1}) - \psi(F_k) - \left(  \psi(\tilde F_k) - \psi(\tilde F_{k-1}) \right) \\
                        &                         & = \phantom{-} & \psi(F_{k+1}) - \psi(F_k) -  \psi(\tilde F_k) + \psi(\tilde F_{k-1})  \\
    \intertext{and}
    (\ve v - \ve w)_{d} & =  v_{d} -  w_{d} & = \phantom{-} & \psi(F_{k}) - \psi(F_{k-1}) - \left(  \psi(\tilde F_{k+1}) - \psi(\tilde F_{k}) \right) \\
                        &                         & = \phantom{-} & \psi(F_{k}) - \psi(\tilde F_{k-1}) - \left(  \psi(F_{k+1}) - \psi(\tilde F_{k}) \right) \\
                        &                         & = - & \psi(F_{k+1}) + \psi(F_{k}) +  \psi(\tilde F_{k}) - \psi(\tilde F_{k-1}).    
\end{align*}
Therefore $(\ve v - \ve w)_{c} = - (\ve v - \ve w)_{d}$ and $\ve w - \ve v$ is
    parallel to $\ve e_d - \ve e_c$. Moreover, $c \in \supp(\ve v)$ since $\ve w, \ve v \geq \ve 0$
    by assumption that $\ve v, \ve w \in \Po(\psi)$. Thus $\ve w - \ve v$, a
    generator of $\Co$, is parallel to a vector in $R_{\supp(\ve v)}$.
\end{proof}

\subsection{The construction of the multivariate rational generating function}

From the knowledge of triangulations of tangent cones of matroid
polytopes (and independence matroid polytopes) we will recover the
multivariate generating functions. 
The following lemma is due to \cite{Brion88} and independently \cite{lawrence91-2}.
A proof can also be found in
\cite{barvinok:99} and \cite{beck-haase-sottile:theorema}.

\begin{lemma}[Brion--Lawrence's Theorem] \label{lem:bl}
  Let $\Po$ be a rational polyhedron
  and $V(\Po)$ be the set of vertices of~$\Po$. Then,
  \begin{equation*}
    g_\Po(\mathbf{z}) = \sum_{\ve v \in V(\Po)} g_{\Co_\Po(\ve v)}(\mathbf{z}),
  \end{equation*}
  where $\Co_\Po(\ve v)$ is the tangent cone of~$\ve v$.
\end{lemma}
Thus, we can write the multivariate generating function of $\Po$ by writing
all multivariate generating functions of all the tangent cones of the vertices
of~$\Po$.
Moreover, the map assigning to a rational polyhedron~$\Po$ its multivariate
rational generating function~$g_\Po(\ve z)$ is a \emph{valuation}, i.e., a
finitely additive measure, so it satisfies the equation
\begin{equation*}
  g_{\Po_1\cup \Po_2}(\ve z) = g_{\Po_1} + g_{\Po_2} - g_{\Po_1\cap \Po_2},
\end{equation*}
for arbitrary rational polytopes $\Po_1$~and~$\Po_2$,
the so-called \emph{inclusion--exclusion principle}.  
This allows to break a polyhedron~$\Po$ into pieces $\Po_1$~and~$\Po_2$ and to compute
the multivariate rational generating functions for the pieces (and their
intersection) separately in order to get the generating function~$g_{\Po}$. 
More generally, let us denote by~$[\Po]$ the \emph{indicator function} of~$\Po$,
i.e., the function
\begin{equation*}
  [\Po]\colon \R^n\to\R, \quad
  [\Po](\ve x) = 
  \begin{cases}
    1 & \text{if $\ve x\in \Po$} \\
    0 & \text{otherwise}.
  \end{cases}
\end{equation*}
Let
\begin{math}
  \sum_{i\in I} \epsilon_i [\Po_i] = 0
\end{math}
be an arbitrary linear identity of indicator functions of rational polyhedra
(with rational coefficients~$\epsilon_i$);
the valuation property now implies that it carries over to a linear identity 
\begin{math}
  \sum_{i\in I} \epsilon_i\, g_{\Po_i}(\ve z) = 0
\end{math}
of rational generating functions.  

Now let $\Co$ be one of the tangent cones of~$\Po$, and let $\Tr$ be a
triangulation of~$\Co$, given by its simplicial cones of maximal dimension.  Let
$\hat \Tr$ denote the set of all (lower-dimensional) intersecting proper faces
of the cones~$\Co_i\in\Tr$.  Then we can assign an integer
coefficient~$\epsilon_i$ to every cone $\Co_i \in \hat\Tr$, such that the
following identity holds:
\begin{equation*}
  \label{eq:indicator-identity-triang}
  [\Co] = \sum_{\Co_i\in \Tr} [\Co_i] + \sum_{\Co_i\in \hat \Tr} \epsilon_i\, [\Co_i].
\end{equation*}
This identity immediately carries over to an identity of multivariate
rational generating functions,
\begin{equation}
  \label{eq:ratgenfun-identity-triang}
  g_\Co(\ve z) = \sum_{\Co_i\in \Tr} g_{\Co_i}(\ve z) 
  + \sum_{\Co_i\in \hat \Tr} \epsilon_i\, g_{\Co_i}(\ve z).
\end{equation}
\begin{remark}
  Notice that
  formula~\eqref{eq:ratgenfun-identity-triang} is of exponential size, even
  when the triangulation~$\Tr$ only has polynomially many simplicial cones
  of maximal dimension.
  The reason is that, when the dimension~$n$ is
  allowed to vary, there are exponentially many intersecting proper faces in
  the set~$\hat\Tr$.  Therefore, we cannot
  use~\eqref{eq:ratgenfun-identity-triang} to compute the multivariate
  rational generating function of~$\Co$ in polynomial time for varying dimension.
\end{remark}

To obtain a shorter formula, we use the technique of \emph{half-open exact
  decompositions}~\cite{koeppe-verdoolaege:parametric}, which is a
refinement of the method of ``irrational'' perturbations
\cite{beck-sottile:irrational,koeppe:irrational-barvinok}.  We use the
following result; see also \autoref{fig:2d-triang} and
\autoref{fig:2d-triang-halfopen-good}.  
\begin{lemma}~
  \label{th:exactify-identities}
  \begin{enumerate}[\rm(a)]
  \item\label{th:exactify-identities-general-part} Let
    \begin{equation}
      \label{eq:full-source-identity}
      \sum_{i\in I_1} \epsilon_i [\Co_i] + \sum_{i\in I_2} \epsilon_i [\Co_i] = 0
    \end{equation}
    be a linear identity (with rational coefficients~$\epsilon_i$) of
    indicator functions of cones~$\Co_i\subseteq\R^n$, where the cones~$\Co_i$
    are full-dimensional for $i\in I_1$ and lower-dimensional for $i\in I_2$.
    Let each cone be given as
    \begin{align}
      \Co_i &= \bigl\{\, \ve x \in\R^n : \langle \ve b^*_{i,j}, \ve x\rangle
      \leq 0 \text{ for $j\in J_i$}\,\bigr\}.
    \end{align}
    Let $\ve y\in\R^n$ be a vector such that $\langle \ve b^*_{i,j}, \ve
    y\rangle \neq 0$ for all $i\in I_1\cup I_2$, $j\in J_i$.  For $i\in I_1$,
    we define the ``half-open cone''
    \begin{equation}
      \label{eq:half-open-by-y}
      \begin{aligned}
        \tilde \Co_i = \Bigl\{\, \ve x\in\R^d : {}& \langle \ve b^*_{i,j}, \ve
        x\rangle \leq 0
        \text{ for $j\in J_i$ with $\langle \ve b^*_{i,j}, \ve y \rangle < 0$,} \\
        & \langle \ve b^*_{i,j}, \ve x\rangle < 0 \text{ for $j\in J_i$ with
          $\langle \ve b^*_{i,j}, \ve y \rangle > 0$} \,\Bigr\}.
      \end{aligned}
    \end{equation}
    Then
    \begin{equation}
      \label{eq:target-identity}
      \sum_{i\in I_1} \epsilon_i [\tilde \Co_i] = 0.
    \end{equation}
\item\label{th:exactify-identities-triang-part}
  In particular, let 
  \begin{equation}
    \label{eq:source-identity-triang}
    [\Co] = \sum_{\Co_i\in \Tr} [\Co_i] + \sum_{\Co_i\in \hat \Tr} \epsilon_i\, [\Co_i]
  \end{equation}
  be the identity corresponding to a triangulation of the cone~$\Co$, where
  $\Tr$ is the set of simplicial cones of maximal dimension and $\hat \Tr$ is
  the set of intersecting proper faces.
  Then there exists a polynomial-time algorithm to construct a vector~$\ve
  y\in\Q^n$ such that the above construction yields the identity
  \begin{equation}
    \label{eq:half-open-triang}
    [\Co] = \sum_{\Co_i\in \Tr} [\tilde \Co_i],
  \end{equation}
  which describes a partition of $\Co$ into half-open cones of maximal dimension.
  \end{enumerate}
\end{lemma}
\begin{proof}
  Part~(\ref{th:exactify-identities-general-part}) is a slightly less general
  form of Theorem~3 in \cite{koeppe-verdoolaege:parametric}.
  Part~(\ref{th:exactify-identities-triang-part}) follows from the discussion
  in section~2 of  \cite{koeppe-verdoolaege:parametric}.
\end{proof}

\begin{figure}[t]
  \centering
  \ifpdf
    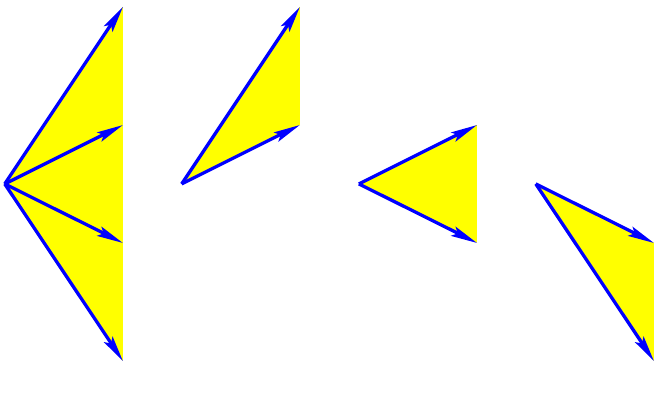
    \else
    \input{2d-triang.pstex_t}
    \fi
  \caption{An identity, valid modulo lower-dimensional cones, corresponding to
    a polyhedral subdivision of a cone}
  \label{fig:2d-triang}
\end{figure}
\begin{figure}[t]
  \ifpdf
    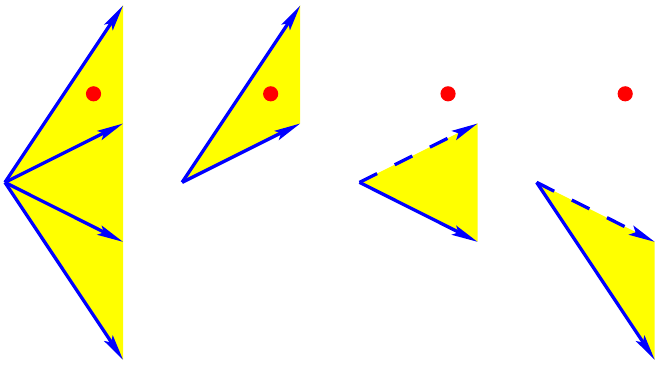
    \else
    \input{2d-triang-halfopen-good.pstex_t}
    \fi
  \caption{The technique of half-open exact decomposition.  The relative
    location of the vector $\ve y$ (represented by a dot) determines which defining
    inequalities are strict (broken lines) and which are weak (solid lines).}
  \label{fig:2d-triang-halfopen-good}
\end{figure}


Since the cones in a triangulation~$\Tr$ of all tangent cones~$\Co_\Po(\ve v)$
of our polytopes are unimodular by \autoref{lem:eluni}, we can efficiently
write the multivariate generating functions of their half-open counterparts.
\begin{lemma}[Lemma~9 in \cite{koeppe-verdoolaege:parametric}] \label{th:unimodular-formula}
  Let $\tilde \Co\subseteq\R^n$ be an $N$-dimensional half-open pointed
  simplicial affine cone with an integral apex~$\ve v\in\Z^n$ and the
  ray description 
  \begin{align}
    \label{eq:explicit}
    \tilde \Co&= \Bigl\{\, \ve v + \textstyle\sum_{j=1}^N \lambda_j \ve b_j :{} 
    \lambda_j \geq 0 \text{ for $j\in J_\leq$ and} \ 
    \lambda_j > 0 \text{ for $j\in J_<$} 
    \,\Bigr\}
  \end{align}
  where $J_\leq\cup J_< = \{1,\dots,N\}$ and $\ve b_j\in\Z^n\setminus\{\ve0\}$.
  We further assume that $\tilde \Co$ is unimodular, i.e., the vectors~$\ve
  b_j$ form a basis of the lattice $(\R\ve b_1+\dots+\R\ve b_N)\cap\Z^n$.
  Then the unique point in the fundamental parallelepiped of the half-open
  cone~$\tilde \Co$ is
  \begin{equation}
    \ve a = \ve v + \sum_{j\in J_<} \ve b_j,
  \end{equation}
  and the generating function of $\Co$ is given by
  \begin{equation}
    g_\Co(\ve{z}) = \frac{\ve{z}^{\ve a}}{\prod_{j=1}^N (1 - \ve{z}^{\ve b_j})}.
  \end{equation}
\end{lemma}
\smallbreak

Taking all results together, we obtain:
\begin{corollary}
  \label{th:matroid-multiratgenfun-summary}
  Let $r$ be a fixed integer.  There exist algorithms that, given 
  \begin{enumerate}[(a)]
  \item a matroid $M$ on $n$ elements, presented by an
    evaluation oracle for its rank function~$\varphi$, which is bounded above
    by~$r$, or
  \item an evaluation oracle for an integral polymatroid rank function $\psi \colon 2^{[n]}
    \longrightarrow \mathbb{N}$, which is bounded above by~$r$,
  \end{enumerate}
  compute in time polynomial in~$n$ vectors $\ve a_i\in \Z^n$, $\ve
  b_{i,j}\in\Z^n\setminus\{\ve0\}$, and $\ve v_i\in\Z^n$ for $i\in I$ (a
  polynomial-size index set) and $j=1,\dots,N$, where $N\leq n$, such that
  the multivariate generating function of $\Po(M)$, $\Po^\In(M)$ and
  $\Po(\psi)$, respectively, is the sum of rational functions
  \begin{align}
    g_{\Po}(\ve z) &= \sum_{i\in I} \frac{\ve z^{\ve a_i}} {\prod_{j=1}^N
      (1-\ve z^{\ve b_{i,j}})}
  \end{align}
  and  the $k$-th dilation of the polytope has the multivariate rational
  generating function 
  \begin{align}
    g_{k\Po}(\ve z) &= \sum_{i\in I} \frac{\ve z^{\ve a_i+(k-1)\ve v_i}}
    {\prod_{j=1}^N (1-\ve z^{\ve b_{i,j}})}.
  \end{align}
\end{corollary}
\begin{proof}
    Lemma \ref{lem:bl} implies that finding the multivariate generating
    function of $\Po(M)$, $\Po^\In(M)$ or $\Po(\psi)$ can be reduced to
    finding the multivariate generating functions of their tangent cones.
    Moreover, $\Po(M)$, $\Po^\In(M)$ and $\Po(\psi)$ have only polynomially in
    $n$ many vertices as described in subsection \ref{subsec:ontangent} or
    Lemma \ref{lem:polyvert}. Enumerating their vertices can be done in
    polynomial time by Lemma \ref{lem:adj} or \ref{lem:polyenum}. 

    Given a vertex $\ve v$ of $\Po(M)$, $\Po^\In(M)$ or $\Po(\psi)$, its neighbors
    can be computed in polynomial time by Lemma \ref{lem:adj} or
    \ref{lem:polyenum}. The tangent cone $\Co(\ve v)$ at $\ve v$ is generated by
    elements in $R_A$, where $R_A$ is an elementary set for some $A \subseteq [n]$.
    See subsection \ref{subsec:ontangent} or Lemma \ref{lem:polyelem}. We also proved
    in Lemma \ref{lem:eluni} that every triangulation of $\Co(\ve v)$ generated by elements in
    $R_A$ is unimodal and Lemma \ref{lem:elsimpbnd} states that any
    triangulation of $\Co(\ve v)$ has at most a polynomial in $n$ number of top-dimensional
    simplices. Moreover, a triangulation of the cone $\Co(\ve v)$ can be computed
    in polynomial time by Lemma \ref{thm:ept}. Finally, using Lemmas
    \ref{th:exactify-identities} and \ref{th:unimodular-formula} we can write
    the polynomial sized multivariate generating function of $\Co(\ve v)$ in polynomial time.
    Therefore we can write the multivariate generating function of $\Po(M)$,
    $\Po^\In(M)$ or $\Po(\psi)$ in polynomial time.
\end{proof}

\subsection{Polynomial-time specialization of rational generating functions in
  varying dimension}

We now compute the Ehrhart polynomial $i(\Po, k) = \#(k\Po\cap\Z^n)$
from the multivariate rational generating function $g_{k\Po}(\ve z)$
of \autoref{th:matroid-multiratgenfun-summary}.  This amounts to the
problem of evaluating or \emph{specializing} a rational generating
function $g_{k\Po}(\ve z)$, depending on a parameter~$k$, at the
point~$\ve z=\ve 1$. This is a pole of each of its summands but a
regular point (removable singularity) of the function itself.  From
now on we call this the \emph{specialization problem}. We explain a
very general procedure to solve it which we hope will allow future
applications.

To this end,
let the generating function of a polytope~$\Po\subseteq\R^n$ be given in the form
\begin{equation}
  \label{eq:generating-function-0}
  g_\Po(\ve z) = \sum_{i\in I} \epsilon_i \frac{\ve
    z^{\ve a_i}}{\prod_{j=1}^{s_i} (1-\ve z^{\ve b_{ij}})}
\end{equation}
where $\epsilon_i\in\{\pm1\}$, $\ve a_i\in\Z^n$, and $\ve
b_{ij}\in\Z^n\setminus\{\ve0\}$.  Let $s = \max_{i\in I} s_i$ be the maximum
number of binomials in the denominators.  In general, if $s$ is allowed to
grow, more poles need to be considered for each summand, so the evaluation
will need more computational effort.

In previous literature, the specialization problem has been considered, but not in
sufficient generality for our purpose.
In the original paper by~\citet[Lemma~4.3]{bar}, the dimension~$n$ is
fixed, and each summand has exactly $s_i=n$ binomials in the denominator.  The
same restriction can be found in the survey by~\citet{barvinok:99}.  In the more
general algorithmic theory of monomial substitutions developed by
\citet{barvinok-woods-2003,Woods:thesis}, there is no assumption on the
dimension~$n$, but the number~$s$ of binomials in the denominators is fixed.  The
same restriction appears in the paper by~\citet[Lemma
2.15]{verdoolaege-woods-2005}.  In a recent paper, \citet[section~5]{barvinok-2006-ehrhart-quasipolynomial} 
gives a polynomial-time algorithm for the specialization problem for rational
functions of the form 
\begin{equation}
  \label{eq:generating-function-with-exponents}
  g(\ve z) = \sum_{i\in I} \epsilon_i \frac{\ve
    z^{\ve a_i}}{\prod_{j=1}^{s} {(1-\ve z^{\ve b_{ij}})}^{\gamma_{ij}}}
\end{equation}
where the dimension~$n$ is fixed, the number~$s$ of different binomials in
each denominator equals~$n$, but 
the multiplicity $\gamma_{ij}$ is varying.  

We will show that the technique from \citet[section~5]{barvinok-2006-ehrhart-quasipolynomial} can be implemented
in a way such that we obtain a polynomial-time algorithm even for the case of
a general formula~\eqref{eq:generating-function-0},
when the dimension and the number of binomials are allowed to grow.

\begin{theorem}[Polynomial-time specialization]
  \label{th:specialization-polytime-varydim}
  \begin{enumerate}[\rm(a)]
  \item There exists an algorithm for computing the
    specialization of a rational function of the form
    \begin{equation}
      \label{eq:generating-function-01}
      g_\Po(\ve z) = \sum_{i\in I} \epsilon_i \frac{\ve
        z^{\ve a_i}}{\prod_{j=1}^{s_i} (1-\ve z^{\ve b_{ij}})}
    \end{equation}
    at its removable singularity~$\ve z=\ve 1$, 
    which runs in time polynomial in the encoding size of its data
    $\epsilon_i\in\Q$, $\ve a_i\in\Z^n$ for $i\in I$ and 
    $\ve b_{ij}\in\Z^n$ for $i\in I$, $j=1,\dots,s_i$, 
    even when the dimension~$n$ and
    the numbers~$s_i$ of terms in the denominators are not fixed.
  \item In particular, there exists a polynomial-time algorithm that, given 
    data $\epsilon_i\in\Q$, $\ve a_i\in\Z^n$ for $i\in I$ and 
    $\ve b_{ij}\in\Z^n$ for $i\in I$, $j=1,\dots,s_i$ describing a rational
    function in the form~\eqref{eq:generating-function-01},  
    computes 
    a vector $\velambda\in\Q^n$ with $\inner{\velambda, \ve b_{ij}} \neq
    0$ for all $i,j$ 
    and rational weights $w_{i,l}$
    for $i\in I$ and $l=0,\dots,s_i$. 
    Then the number of integer points is given by
    \begin{equation}
      \#(\Po\cap\Z^n) 
      = \sum_{i\in I} \epsilon_i \sum_{l=0}^{s_i} w_{i,l} \inner{\velambda, \ve a_i}^l.
    \end{equation}
  \item Likewise, given a parametric rational function for the dilations of an
    integral polytope~$\Po$, 
    \begin{align}\label{eq:generating-function-dil}
      g_{k\Po}(\ve z) 
      &= \sum_{i\in I} \epsilon_i \frac{\ve z^{\ve a_i+(k-1)\ve v_i}} 
      {\prod_{j=1}^d (1-\ve z^{\ve b_{i,j}})},
    \end{align}
    the Ehrhart polynomial $i(\Po, k) = \#(k\Po\cap\Z^n)$ is given by the explicit formula
    \begin{equation}
      i(\Po, k) = 
      \sum_{m=0}^M \paren{
        \sum_{i\in I} 
        \epsilon_i 
        \inner{\velambda,\ve v_i}^m 
        \sum_{l=m}^{s_i} \binom{l}{m} w_{i,l} 
        \inner{\velambda,\ve a_i - \ve v_i}^{l-m} 
      } k^m,
    \end{equation}
    where $M = \min\{s,\dim\Po\}$.
  \end{enumerate}
\end{theorem}

\begin{proof} [Proof of Theorem \ref{volume}] 
Corollary \ref{th:matroid-multiratgenfun-summary} and
Theorem \ref{th:specialization-polytime-varydim} imply Theorem
(\autoref{volume}) directly. 
\end{proof}

The remainder of this section contains the proof of
\autoref{th:specialization-polytime-varydim}.  We follow \cite{barvinok:99} and
recall the definition of Todd polynomials. We will prove that they can be
efficiently evaluated in rational arithmetic.
\begin{definition}\label{def:todd-poly}
  We consider the function
  \begin{displaymath}
    H(x, \xi_1,\dots,\xi_s) = \prod^s_{i=1} \frac{x\xi_i}{1-\exp\{-x\xi_i\}},
  \end{displaymath}
  a function that is analytic in a neighborhood of~$\ve 0$.
  The $m$-th ($s$-variate) \emph{Todd polynomial} is the coefficient of~$x^m$
  in the Taylor expansion 
  \begin{displaymath}
    H(x, \xi_1,\dots,\xi_s) = \sum_{m=0}^\infty \td_m(\xi_1,\dots,\xi_s) x^m.
  \end{displaymath}
\end{definition}
We remark that, when the numbers $s$ and $m$ are allowed to vary, the Todd
polynomials have an exponential number of monomials.
\begin{theorem}\label{th:todd-evaluation}
  The Todd polynomial $\td_m(\xi_1,\dots,\xi_s)$ can be evaluated for given
  rational data $\xi_1,\dots,\xi_s$ in time polynomial in~$s$, $m$, and the
  encoding length of $\xi_1,\dots,\xi_s$.
\end{theorem}
The proof makes use of the following lemma.
\begin{lemma}\label{lemma:toddy-expansion}
  The function $h(x) = x/(1-\mathrm{e}^{-x})$ is a function that is analytic in a
  neighborhood of~$0$.  Its Taylor series about $x=0$ is of the form 
  \begin{equation}
    \label{eq:toddy-coefficients-format}
    h(x) = \sum_{n=0}^\infty b_n x^n \quad\text{where}\quad
    b_n = \frac1{n!\, (n+1)!} c_n
  \end{equation}
  with integer coefficients $c_n$ that have a binary encoding length of
  $\mathrm O(n^2 \log n)$.
  The coefficients $c_n$ can be computed from the recursion
  \begin{equation}
    \label{eq:recursion-toddy}
    \begin{aligned}
      c_0 &= 1\\
      c_n &= \sum_{j=1}^n (-1)^{j+1} \binom{n+1}{j+1} \frac{n!}{(n-j+1)!} c_{n-j}
      && \text{for $n=1,2,\dots$.}
    \end{aligned}
  \end{equation}
\end{lemma}
\begin{proof}
  The reciprocal function $h^{-1}(x) = (1-\mathrm{e}^{-x})/x$ has the Taylor
  series
  \begin{displaymath}
    h^{-1}(x) = \sum_{i=0}^\infty a_n x^n
    \quad\text{with}\quad
    a_n = \frac{(-1)^n}{(n+1)!}.
  \end{displaymath}
  Using the identity $h^{-1}(x)h(x) = \bigl( \sum_{n=0}^\infty a_n x^n \bigr)
  \bigl( \sum_{n=0}^\infty b_n x^n \bigr) = 1$, we obtain the recursion
  \begin{equation}
    \begin{aligned}
      b_0 &= \tfrac1{a_0} = 1\\
      b_n &= -(a_1 b_{n-1} + a_2 b_{n-2} + \dots + a_n b_0)
      && \text{for $n=1,2,\dots$.}
    \end{aligned}
  \end{equation}
  We prove~\eqref{eq:toddy-coefficients-format} inductively.  Clearly $b_0 =
  c_0 = 1$.  For $n=1,2,\dots$, we have
  \begin{align*}
    c_n &= n! \, (n + 1)! \, b_n \\
    &= -n! \, (n + 1)! \, (a_1 b_{n-1} + a_2 b_{n-2} + \dots + a_n b_0) \\
    &= n! \, (n + 1)! \, \sum_{j=1}^n \frac{(-1)^{j+1}}{(j+1)!} \cdot
    \frac{1}{(n-j)!\, (n-j+1)!} c_{n-j} \\
    &= \sum_{j=1}^n (-1)^{j+1} \frac{(n+1)!}{(j+1)!\, (n-j)!} \cdot 
    \frac{n!}{(n-j+1)!} c_{n-j}.
  \end{align*}
  Thus we obtain the recursion formula~\eqref{eq:recursion-toddy}, which also
  shows that all $c_n$ are integers.  A rough estimate shows that 
  \begin{displaymath}
    \abs{c_n} \leq n (n+1)!\, n!\, \abs{c_{n-1}} \leq \paren[big]{(n+1)!}^2 \abs{c_{n-1}},
  \end{displaymath}
  thus $\abs{c_n}\leq \paren[big]{(n+1)!}^{2n}$, so $c_n$ has a binary
  encoding length of $\mathrm O(n^2 \log n)$. 
\end{proof}

\begin{proof}[Proof of \autoref{th:todd-evaluation}]
  By definition, we have
  \begin{displaymath}
    H(x,\xi_1,\dots,\xi_s) = \sum_{m=0}^\infty \td_m(\xi_1,\dots,\xi_s) x^m 
    = \prod_{j=1}^s h(x\xi_j).
  \end{displaymath}
  From \autoref{lemma:toddy-expansion} we have
  \begin{equation}
    \label{eq:toddy-expansion-trunc}
    h(x\xi_j) = \sum_{n=0}^m \beta_{j,n} x^n  + o(x^m)
    \quad\text{where}\quad
    \beta_{j,n} = \frac{\xi_j^n}{n!\,(n+1)!} c_n 
  \end{equation}
  with integers~$c_n$ given by the recursion~\eqref{eq:recursion-toddy}.
  Thus we can evaluate $\td_m(\xi_1,\dots,\xi_s)$  by summing over all the
  possible compositions $n_1 + \dots + n_{s} = m$ of the order~$m$
  from the orders~$n_j$ of the factors:
  \begin{equation}\label{eq:evaluation-by-compositions}
    \td_m(\xi_1,\dots,\xi_s)
    = \sum_{\substack{(n_1,\dots,n_{s})\in\Z_+^{s}\\
        n_1 + \dots + n_{s} = m}} \!\!\!\! \beta_{1,n_1} \dots
    \beta_{s,n_s}
  \end{equation}
  We remark that the length of the above sum is equal to the number of compositions 
  of~$m$ into $s$~non-negative parts, 
  \begin{align*}
    C'_s(m) 
    &= \binom{m + s - 1}{s - 1}\\
    &= \frac{(m+s-1)(m+s-2)\dots(m+s-(s-1))} {(s-1) (s-2) \dots 2\cdot 1}\\
    &= \Omega\paren{\paren[big]{1 + \tfrac{m}{s-1}}^{s}},
  \end{align*}
  which is \emph{exponential} in~$s$ (whenever $m\geq s$).  Thus we cannot evaluate the
  formula~\eqref{eq:evaluation-by-compositions}
  efficiently when $s$ is allowed to grow.  

  However, we show that we can evaluate $\td_m(\xi_1,\dots,\xi_s)$ more efficiently. 
  To this end, we multiply up the $s$ truncated Taylor
  series~\eqref{eq:toddy-expansion-trunc}, one
  factor at a time, truncating after order~$m$.  Let us denote
  \begin{align*}
    H_1(x) &= h(x\xi_1) \\
    H_2(x) &= H_1(x) \cdot h(x\xi_2) \\
    &\vdots \\
    H_s(x) &= H_{s-1}(x) \cdot h(x\xi_s) = H(x,\xi_1,\dots,\xi_s).
  \end{align*}
  Each multiplication
  can be implemented in $\mathrm O(m^2)$ elementary rational operations.  
  We finally show that all numbers appearing in the calculations have polynomial
  encoding size.  Let $\Xi$ be the largest binary encoding size of any of the rational
  numbers~$\xi_1,\dots,\xi_s$.  Then every $\beta_{j,n}$ given
  by~\eqref{eq:toddy-expansion-trunc} has a binary encoding 
  size $\mathrm O(\Xi n^5 \log^3 n)$.  Let $H_j(x)$ have the truncated Taylor
  series $\sum_{n=0}^m \alpha_{j,n} x^n + o(x^m)$ and let $A_j$ denote the
  largest binary encoding length of any $\alpha_{j,n}$ for $n\leq m$.  Then 
  \begin{displaymath}
    H_{j+1}(x) = \sum_{n=0}^m \alpha_{j+1,n} x^n + o(x^m)
    \quad\text{with}\quad
    \alpha_{j+1,n} =  \sum_{l=0}^n \alpha_{j,l} \beta_{j, n-l} .
  \end{displaymath}
  Thus the binary encoding size of $\alpha_{j+1,n}$ (for $n\leq m$) is bounded
  by $A_j + \mathrm O(\Xi m^5 \log^3 m)$. 
  Thus, after $s$ multiplication steps, the encoding size of the coefficients
  is bounded by $\mathrm O(s \Xi m^5 \log^3 m)$, a polynomial quantity.
\end{proof}

\begin{proof}[Proof of \autoref{th:specialization-polytime-varydim}]
\emph{Parts (a) and (b).}
  We recall the technique of \citet[Lemma~4.3]{bar}, refined by
  \citet[section~5]{barvinok-2006-ehrhart-quasipolynomial}.
  
  We first construct a rational vector $\velambda\in\Z^n$ such that
  $\inner{\velambda, \ve b_{ij}}\neq0$ for all $i,j$.  One such construction
  is to consider the \emph{moment curve} $\velambda(\xi) = (1, \xi, \xi^2,
  \dots, \xi^{n-1})\in\R^n$.  For each exponent vector $\ve b_{ij}$ occuring
  in a denominator of~\eqref{eq:generating-function-0}, the function
  $f_{ij}\colon \xi\mapsto \inner{\velambda(\xi), \ve b_{ij}}$ is a
  polynomial function of degree at most $n-1$.  Since $\ve b_{ij}\neq\ve 0$,
  the function $f_{ij}$ is not identically zero.  Hence $f_{ij}$ has at most
  $n-1$ zeros.  By evaluating all functions $f_{ij}$ for $i\in I$ and
  $j=1,\dots,s_i$ at $M = (n-1)s|I|+1$ different values for~$\xi$, for
  instance at the integers $\xi=0,\dots, M$, we can find one~$\xi=\bar\xi$
  that is not a zero of any~$f_{ij}$.  Clearly this search can be implemented
  in polynomial time, even when the dimension~$n$ and the number~$s$ of terms
  in the denominators are not fixed.  We set $\velambda =
  \velambda(\bar\xi)$.
  
  For $\tau>0$, let us consider the points $\ve z_\tau = \ve e^{\tau
    \velambda} = (\exp\{ \tau\lambda_1 \}, \dots, \exp\{ \tau\lambda_n
  \})$. We have 
  \begin{displaymath}
    \ve z_\tau^{\ve b_{ij}} 
    = \prod_{l=1}^n \exp\{ \tau\lambda_l b_{ijl} \} = \exp\{ \tau
    \inner{\velambda, \ve b_{ij}} \};
  \end{displaymath}
  since $\inner{\velambda, \ve b_{ij}}\neq0$ for all $i,j$, all the
  denominators $1-\ve z_\tau^{\ve b_{ij}}$ are nonzero.  
  Hence for every
  $\tau>0$, the point $\ve z_\tau$ is a regular point not only of $g(\ve z)$
  but also of the individual summands of~\eqref{eq:generating-function-0}. 
  We have
  \begin{align*}
    g(\ve 1) 
    &= \lim_{\tau\to0^+} \sum_{i\in I} \epsilon_i \frac{\ve
      z_\tau^{\ve a_i}}{\prod_{j=1}^{s_i} (1-\ve z_\tau^{\ve b_{ij}})} 
\displaybreak[0]\\
    &= \lim_{\tau\to0^+} \sum_{i\in I} \epsilon_i 
    \frac{\exp\{ \tau \inner{\velambda, \ve a_{i}} \}}
    {\prod_{j=1}^{s_i} (1-\exp\{ \tau \inner{\velambda, \ve b_{ij}} \} ) } 
\displaybreak[0]\\
    &= \lim_{\tau\to0^+} \sum_{i\in I} \epsilon_i \,
    \tau^{-s_i}
    \exp\{ \tau \inner{\velambda, \ve a_{i}} \}
    \prod_{j=1}^{s_i}
    \frac{\tau} 
    {1-\exp\{ \tau \inner{\velambda, \ve b_{ij}} \} } 
\displaybreak[0]\\
    &= \lim_{\tau\to0^+} \sum_{i\in I} \epsilon_i \,
    \tau^{-s_i}
    \exp\{ \tau \inner{\velambda, \ve a_{i}} \}
    \prod_{j=1}^{s_i}
    \frac{-1}{\inner{\velambda, \ve b_{ij}}} h(-\tau\inner{\velambda, \ve
      b_{ij}}) 
\displaybreak[0]\\
    &= \lim_{\tau\to0^+} \sum_{i\in I} \epsilon_i\,  \frac{(-1)^{s_i}}{\prod_{j=1}^{s_i}
      \inner{\velambda, \ve b_{ij}}}  
    \,    \tau^{-s_i}
    \exp\{ \tau \inner{\velambda, \ve a_{i}} \}
    H(\tau, -\inner{\velambda, \ve b_{i1}}, \dots, -\inner{\velambda, \ve b_{is_i}})
  \end{align*}
  where $H(x, \xi_1,\dots,\xi_{s_i})$ is the function
  from \autoref{def:todd-poly}. 
  We will compute the limit 
  by computing the constant term of the Laurent expansion of each summand
  about~$\tau=0$. 
  Now the function $\tau\mapsto \exp\{ \tau \inner{\velambda, \ve a_{i}} \}$
  is holomorphic and has the Taylor series
  \begin{equation}
    \exp\{ \tau \inner{\velambda, \ve a_{i}} \}
    = \sum_{l = 0}^{s_i} \alpha_{i,l} \tau^l + o(\tau^{s_i})
    \quad\text{where}\quad
    \alpha_{i,l} = \frac{\inner{\velambda, \ve a_{i}}^l}{l!}, 
    \label{eq:numerator-series}
  \end{equation}
  and $H(\tau, \xi_1,\dots,\xi_{s_i})$ has the Taylor series
  \begin{displaymath}
    H(\tau, \xi_1,\dots,\xi_s) = \sum_{m=0}^{s_i} \td_m(\xi_1,\dots,\xi_s)
    \tau^m + o(\tau^{s_i}). 
  \end{displaymath}
  Because of
  the factor $\tau^{-s_i}$, which gives rise to a pole of order~$s_i$ in the
  summand, we can compute the constant term of the
  Laurent expansion by summing over all the possible compositions $s_i = l +
  (s_i - l)$ of the order~$s_i$:
  \begin{equation}
    g(\ve 1) 
    = \sum_{i\in I} \epsilon_i 
    \frac{(-1)^{s_i}}{\prod_{j=1}^{s_i} \inner{\velambda, \ve b_{ij}}}
    \sum_{l=0}^{s_i} \frac{\inner{\velambda, \ve a_i}^l}{l!} 
    \td_{s_i-l}(-\inner{\velambda, \ve b_{i1}}, \dots, -\inner{\velambda, \ve b_{is_i}}).
  \end{equation}
  We use the notation
  \begin{displaymath}
    w_{i,l}  =  (-1)^{s_i} \frac{\td_{s_i-l}( -\langle \velambda,
      \ve b_{i,1} \rangle, \dots, -\langle \velambda, \ve b_{i,s_i}\rangle)}
    {l!\cdot \langle \velambda,
      \ve b_{i,1} \rangle \cdots \langle \velambda, \ve b_{i,s_i}\rangle}
    \quad\text{for $i\in I$ and $l=0,\dots,s_i$};
  \end{displaymath}
  these rational numbers can be computed in polynomial time using
  \autoref{th:todd-evaluation}.  
  We now obtain the formula of the claim, 
  \begin{displaymath}
    g(\ve 1) 
    = \sum_{i\in I} \epsilon_i \sum_{l=0}^{s_i} w_{i,l} \inner{\velambda, \ve a_i}^l.
  \end{displaymath}
  \smallbreak

\noindent\emph{Part~(c).}
Applying the same technique to the parametric rational function~\eqref{eq:generating-function-dil}, we obtain
\begin{align*}
  \#(k\Po\cap\Z^n) &= g_{k\Po}(\ve1)\\
  &= \sum_{i\in I} \epsilon_i \sum_{l=0}^{s_i} w_{i,l} \inner{\velambda, \ve a_i+(k-1)\ve v_i}^l, 
\displaybreak[0]\\
  &= \sum_{i\in I} \epsilon_i \sum_{l=0}^{s_i} w_{i,l} \sum_{m=0}^l
  \binom{l}{m} \inner{\velambda,\ve a_i - \ve v_i}^{l-m} k^m \inner{\velambda,\ve
    v_i}^m 
\displaybreak[0]\\
  &= \sum_{m=0}^s \paren{
    \sum_{i\in I} 
    \epsilon_i 
    \inner{\velambda,\ve v_i}^m 
    \sum_{l=m}^{s_i} \binom{l}{m} w_{i,l} 
    \inner{\velambda,\ve a_i - \ve v_i}^{l-m} 
  } k^m,
\end{align*}
an explicit formula for the Ehrhart polynomial.  We remark
that, since the Ehrhart polynomial is of degree equal to the
dimension of~$\Po$, all coefficients of $k^m$ for $m > \dim\Po$ must vanish.
Thus we obtain the formula of the claim, where we sum only up to $\min\{s,
\dim\Po\}$ instead of~$s$. 
\end{proof}


\section{Algebraic Properties of $h^*$-vectors and Ehrhart polynomials of Matroid Polytopes} \label{sec:hstar}

Using the programs {\tt cdd+} \cite{cdd}, {\tt LattE} \cite{lattesoft} and {\tt LattE macchiato} \cite{latte-macchiato} we explored patterns for the
Ehrhart polynomials of matroid polytopes. Since previous authors
proposed other invariants of a matroid (e.g., the Tutte polynomials
and the invariants of \cite{SpeyerA-matroid-invar, billera-2006}) we wished to know how well does the Ehrhart
polynomial distinguish non-isomorphic matroids. It is natural to
compare it with other known invariants. Some straightforward
properties are immediately evident.  For example, the Ehrhart
polynomial of a matroid and that of its dual are equal. Also the
Ehrhart polynomial of a direct sum of matroids is the product of their
Ehrhart polynomials.

\let\maybemidrule=\relax
\begin{table}[htbp]
\caption{Coefficients of the Ehrhart polynomials of selected matroids in \cite{billera-2006, SpeyerA-matroid-invar}}
 \label{tab:belspeyer}
\begin{centering}
\def\arraystretch{1.2}
\begin{tabular}{ll}
\toprule
& Ehrhart Polynomial \\
\midrule
Speyer1 & $1   , \frac{21}{5}   , \frac{343}{45}   , \frac{63}{8}   , \frac{91}{18}   , \frac{77}{40}   , \frac{29}{90}  $ \\
\maybemidrule
Speyer2 & $1   , \frac{135}{28}   , \frac{3691}{360}   , \frac{1511}{120}   , \frac{88}{9}   , \frac{39}{8}   , \frac{529}{360}   , \frac{89}{420}  $\\
\maybemidrule
BJR1 & $1  , \frac{109}{30}   , \frac{23}{4}   , \frac{59}{12}   , \frac{9}{4}   , \frac{9}{20}   $\\
\maybemidrule
BJR2 & $ 1  , \frac{211}{60}   , \frac{125}{24}   , \frac{33}{8}   , \frac{43}{24}   , \frac{43}{120}   $\\
\maybemidrule
BJR3 &  $1   , \frac{83}{20}   , \frac{2783}{360}   , \frac{199}{24}   , \frac{391}{72}   , \frac{247}{120}   , \frac{61}{180}  $\\
\maybemidrule
BJR4 & $1   , \frac{25}{6}   , \frac{353}{45}   , \frac{101}{12}   , \frac{193}{36}   , \frac{23}{12}   , \frac{53}{180}  $\\
\bottomrule
\end{tabular}
\end{centering}
\end{table}

We call the last two matroids in Figure $2$ in \cite{SpeyerA-matroid-invar} {\it Speyer1} and {\it Speyer2} and the  matroids of Figure $2$ in \cite{billera-2006} {\it BJR1}, {\it BJR2}, {\it BJR3}, and {\it BJR4}, and list their Ehrhart polynomials in Table \ref{tab:belspeyer}.
We note that {\it BJR3} and {\it BJR4} 
have the same Tutte polynomial, yet their Ehrhart polynomials are
different. This proves that the Ehrhart polynomial cannot be computed
using deletion and contractions, like is the case for the Tutte
polynomial.  Examples {\it BJR1} and {\it BJR2} show that the Ehrhart
polynomial of a matroid may be helpful on distinguishing
non-isomorphic matroids: These two matroids are not isomorphic yet
they have the same Tutte polynomials and the same quasi-symmetric
function studied in \cite{billera-2006}. Although they share some
properties, there does not seem to be an obvious relation to Speyer's
univariate polynomials introduced in \cite{SpeyerA-matroid-invar};
examples {\it Speyer1} and {\it Speyer2} show they are relatively
prime with their corresponding Ehrhart polynomials.

Our experiments included, among others, many examples coming from
small graphical matroids, random realizable matroids over fields of
small positive characteristic, and the classical examples listed in
the Appendix of \cite{Oxley1992Matroid-Theory} for which we list the
results in Table \ref{tab:hvec}. For a comprehensive list of all our
calculations visit \cite{HawsMatroid-Polytop}.  Soon it became evident
that all instances verified both parts of our \autoref{unimodalconj}.

\begin{table}[htbp]
\caption{Coefficients of the Ehrhart polynomials and $h^*$-vectors of selected
  matroids in \cite{Oxley1992Matroid-Theory}} 
 \label{tab:hvec}
\begin{centering}
\tiny
\def\arraystretch{1.6}
\begin{tabular}{lccll}
\toprule
& \normalsize$n$ & \normalsize$r$ & \normalsize$h^*$-vector & \normalsize Ehrhart Polynomial \\
\midrule
 $K_4$ & $6$ & $3$ & $1, 10, 20, 10, 1$                            & $  1   , \frac{107}{30}   , \frac{21}{4}   , \frac{49}{12}   , \frac{7}{4}   , \frac{7}{20}   $ \\
 \maybemidrule                                                             
 $W^3$ & $6$ & $3$ & $1, 11, 24, 11, 10$                            & $  1   , \frac{18}{5}   , \frac{11}{2}   , \frac{9}{2}   , 2   , \frac{2}{5}   $ \\
 \maybemidrule                                                             
 $Q_6$ & $6$ & $3$ & $1, 12, 28, 12, 1$                            & $  1   , \frac{109}{30}   , \frac{23}{4}   , \frac{59}{12}   , \frac{9}{4}   , \frac{9}{20}   $ \\
 \maybemidrule                                                             
 $P_6$ & $6$ & $3$ & $1, 13, 32, 13, 1$                            & $  1   , \frac{11}{3}   , 6   , \frac{16}{3}   , \frac{5}{2}   , \frac{1}{2}   $ \\
 \maybemidrule                                                             
 $R_6$ & $6$ & $3$ & $1, 12, 28, 12, 1$                            & $  1   , \frac{109}{30}   , \frac{23}{4}   , \frac{59}{12}   , \frac{9}{4}   , \frac{9}{20}   $ \\
 \maybemidrule                                                             
 $F_7$ & $7$ & $3$ & $21, 98, 91, 21, 1$                           & $  1   , \frac{21}{5}   , \frac{343}{45}   , \frac{63}{8}   , \frac{91}{18}   , \frac{77}{40}   , \frac{29}{90}   $ \\
 \maybemidrule                                                             
 $F_7^-$ & $7$ & $3$ & $21, 101, 97, 22, 1$                        & $  1   , \frac{253}{60}   , \frac{2809}{360}   , \frac{33}{4}   , \frac{193}{36}   , \frac{61}{30}   , \frac{121}{360}   $ \\
 \maybemidrule                                                             
 $P^7$ & $7$ & $3$ & $21, 104, 103, 23, 1$                         & $  1   , \frac{127}{30}   , \frac{479}{60}   , \frac{69}{8}   , \frac{17}{3}   , \frac{257}{120}   , \frac{7}{20}   $ \\
 \maybemidrule                                                             
 $AG(3,2)$ & $8$ & $4$ & $1, 62, 561, 1014, 449, 48, 1$            & $  1   , \frac{209}{42}   , \frac{1981}{180}   , \frac{881}{60}   , \frac{119}{9}   , \frac{95}{12}   , \frac{499}{180}   , \frac{89}{210}   $ \\
 \maybemidrule                                                             
 $AG'(3,2)$ & $8$ & $4$ & $1, 62, 562, 1023, 458, 49, 1$           & $  1   , \frac{299}{60}   , \frac{4007}{360}   , \frac{5401}{360}   , \frac{122}{9}   , \frac{2911}{360}   , \frac{1013}{360}   , \frac{77}{180}   $ \\
 \maybemidrule                                                             
 $R_8$ & $8$ & $4$ & $1, 62, 563, 1032, 467, 50, 1$                & $  1   , \frac{524}{105}   , \frac{1013}{90}   , \frac{1379}{90}   , \frac{125}{9}   , \frac{743}{90}   , \frac{257}{90}   , \frac{136}{315}   $ \\
 \maybemidrule                                                             
 $F_8$ & $8$ & $4$ & $1, 62, 563, 1032, 467, 50, 1$                & $  1   , \frac{524}{105}   , \frac{1013}{90}   , \frac{1379}{90}   , \frac{125}{9}   , \frac{743}{90}   , \frac{257}{90}   , \frac{136}{315}   $ \\
 \maybemidrule                                                             
 $Q_8$ & $8$ & $4$ & $1, 62, 564, 1041, 476, 51, 1$                & $  1   , \frac{2099}{420}   , \frac{4097}{360}   , \frac{1877}{120}   , \frac{128}{9}   , \frac{337}{40}   , \frac{1043}{360}   , \frac{61}{140}   $ \\
 \maybemidrule                                                             
 $S_8$ & $8$ & $4$ & $1, 44, 337, 612, 305, 40, 1$                 & $  1   , \frac{1021}{210}   , \frac{377}{36}   , \frac{475}{36}   , \frac{193}{18}   , \frac{511}{90}   , \frac{65}{36}   , \frac{67}{252}   $ \\
 \maybemidrule                                                             
 $V_8$ & $8$ & $4$ & $1, 62, 570, 1095, 530, 57, 1$                & $  1   , \frac{2117}{420}   , \frac{4367}{360}   , \frac{2107}{120}   , \frac{146}{9}   , \frac{1133}{120}   , \frac{1133}{360}   , \frac{193}{420}   $ \\
 \maybemidrule                                                             
 $T_8$ & $8$ & $4$ & $1, 62, 564, 1041, 476, 51, 1$                & $  1   , \frac{2099}{420}   , \frac{4097}{360}   , \frac{1877}{120}   , \frac{128}{9}   , \frac{337}{40}   , \frac{1043}{360}   , \frac{61}{140}   $ \\
 \maybemidrule                                                             
 $V_8^+$ & $8$ & $4$ & $1, 62, 569, 1086, 521, 56, 1$              & $  1   , \frac{151}{30}   , \frac{2161}{180}   , \frac{3103}{180}   , \frac{143}{9}   , \frac{1669}{180}   , \frac{559}{180}   , \frac{41}{90}   $ \\
 \maybemidrule                                                             
 $L_8$ & $8$ & $4$ & $1, 62, 567, 1068, 503, 54, 1$                & $  1   , \frac{527}{105}   , \frac{529}{45}   , \frac{83}{5}   , \frac{137}{9}   , \frac{134}{15}   , \frac{136}{45}   , \frac{47}{105}   $ \\
 \maybemidrule                                                             
 $J$ & $8$ & $4$ & $1, 44, 339, 630, 323, 42, 1$                   & $  1   , \frac{512}{105}   , \frac{193}{18}   , \frac{83}{6}   , \frac{205}{18}   , \frac{361}{60}   , \frac{17}{9}   , \frac{23}{84}   $ \\
 \maybemidrule                                                             
 $P_8$ & $8$ & $4$ & $1, 62, 565, 1050, 485, 52, 1$                & $  1   , \frac{1051}{210}   , \frac{2071}{180}   , \frac{2873}{180}   , \frac{131}{9}   , \frac{1547}{180}   , \frac{529}{180}   , \frac{277}{630}   $ \\
 \maybemidrule                                                             
 $W_4$ & $8$ & $4$ & $1, 38, 262, 475, 254, 37, 1$                 & $  1   , \frac{135}{28}   , \frac{3691}{360}   , \frac{1511}{120}   , \frac{88}{9}   , \frac{39}{8}   , \frac{529}{360}   , \frac{89}{420}   $ \\
 \maybemidrule                                                             
 $W^4$ & $8$ & $4$ & $1, 38, 263, 484, 263, 38, 1$                 & $  1   , \frac{169}{35}   , \frac{467}{45}   , \frac{581}{45}   , \frac{91}{9}   , \frac{227}{45}   , \frac{68}{45}   , \frac{68}{315}   $ \\
 \maybemidrule                                                             
 $K_{3,3}$ & $9$ & $5$ & $78, 1116, 3492, 3237, 927, 72, 1$        & $  1   , \frac{307}{56}   , \frac{137141}{10080}   , \frac{3223}{160}   , \frac{37807}{1920}   , \frac{211}{16}   , \frac{5743}{960}   , \frac{1889}{1120}   , \frac{8923}{40320}   $ \\
 \maybemidrule                                                             
 $AG(2,3)$ & $9$ & $3$ & $1, 147, 1230, 1885, 714, 63, 1$          & $  1   , \frac{1453}{280}   , \frac{41749}{3360}   , \frac{581}{32}   , \frac{34069}{1920}   , \frac{927}{80}   , \frac{4541}{960}   , \frac{239}{224}   , \frac{449}{4480}   $ \\
 \maybemidrule                                                             
 Pappus & $9$ & $3$ & $1, 147, 1230, 1915, 744, 66, 1$             & $  1   , \frac{729}{140}   , \frac{3573}{280}   , \frac{381}{20}   , \frac{1499}{80}   , \frac{243}{20}   , \frac{49}{10}   , \frac{153}{140}   , \frac{57}{560}   $ \\
 \maybemidrule                                                             
 Non-Pappus & $9$ & $3$ & $1, 147, 1230, 1925, 754, 67, 1$         & $  1   , \frac{4379}{840}   , \frac{25951}{2016}   , \frac{9287}{480}   , \frac{21967}{1152}   , \frac{987}{80}   , \frac{2855}{576}   , \frac{3701}{3360}   , \frac{275}{2688}   $ \\
 \maybemidrule                                                             
 $Q_3(\mathit{GF}(3)^*)$ & $9$ & $3$ & $1, 147, 1098, 1638, 632, 59, 1$      & $  1   , \frac{433}{84}   , \frac{3079}{252}   , \frac{4193}{240}   , \frac{5947}{360}   , \frac{167}{16}   , \frac{601}{144}   , \frac{787}{840}   , \frac{149}{1680}   $ \\
 \maybemidrule                                                             
 $R_9$ & $9$ & $3$ & $1, 147, 1142, 1717, 656, 60, 1$              & $  1   , \frac{723}{140}   , \frac{49}{4}   , \frac{88}{5}   , \frac{24217}{1440}   , \frac{1291}{120}   , \frac{625}{144}   , \frac{821}{840}   , \frac{133}{1440}   $ \\
\bottomrule
\end{tabular}
\end{centering}
\end{table}

By far the most comprehensive study we made was for the family of
uniform matroids. In this case we based our computations on the theory
of Veronese algebras as developed by M. Katzman in
\cite{Katzman:math0408038}. There, Katzman gives an explicit
equation for the $h^*$-vector of uniform matroid polytopes (again,
using the language of Veronese algebras). For this family we were
able to verify computationally the conjecture is true for all uniform
matroids up to $75$ elements and to prove partial unimodality as explained
in the introduction.

\begin{lemma} \label{lem:conjb}
The coefficients of the Ehrhart polynomial of the matroid polytope of the uniform matroid $U^{2,n}$ are positive.
\end{lemma}
\begin{proof}
We begin with the expression in Corollary 2.2 in \cite{Katzman:math0408038} which explicitly gives the Ehrhart polynomial of $\Po(U^{r,n})$ as
\begin{equation} \label{eq:uniehrhart}
 i(\Po(U^{r,n}), k) = \sum_{s=0}^{r-1} (-1)^s {n \choose s}{k(r-s) -s +n -1 \choose n-1}.
\end{equation}
Letting $r=2$, Equation \eqref{eq:uniehrhart} becomes
\begin{equation*}
\frac{(2k+n-1)(2k+n-2) \cdots (2k+1)}{(n-1)!} - n \frac{(k+n-2)(k+n-3) \cdots (k)}{(n-1)!}
\end{equation*}
We next consider the coefficient of $k^{n-p-1}$ for $0 \leq p \leq n -1$, which can be written as
\begin{align} 
&2^{n-p-1} \sum_{\substack{1 \leq j_1 < \cdots < j_p \leq n-1 \\ j_q \in \Z }} \frac{1}{(n-j_1) \cdots (n-j_p)} \notag\\
&\qquad - \frac{n}{n-1} \sum_{\substack{1 \leq j_1 < \cdots < j_{p-1} \leq n-2 \\ j_q \in \Z }} \frac{1}{(n-j_1) \cdots(n-j_{p-1})}
\notag
\displaybreak[0]\\
&= 2^{n-p-1} \sum_{\substack{1 \leq j_1 < \cdots < j_{p-1} \leq n-2 \\ j_q \in \Z }} \biggl( \frac{1}{(n-j_1) \cdots (n-j_{p-1})}  \sum_{\substack{j_p=1+j_{p-1}}}^{n-1} \frac{1}{n-j_p} \biggr) \notag\\
&\qquad - \frac{n}{n-1} \sum_{\substack{1 \leq j_1 < \cdots < j_{p-1} \leq n-2 \\ j_q \in \Z }} \frac{1}{(n-j_1) \cdots(n-j_{p-1})}\notag
\displaybreak[0]\\
\label{eq:d2e3}
&=  \sum_{\substack{1 \leq j_1 < \cdots < j_{p-1} \leq n-2 \\ j_q \in \Z }} \biggl( \frac{1}{(n-j_1) \cdots (n-j_{p-1})} \biggl[ 2^{n-p-1} \sum_{\substack{j_p=1 + j_{p-1}}}^{n-1} \frac{1}{n-j_p} - \frac{n}{n-1} \biggr] \biggr).
\end{align}
It is known that the constant in any Ehrhart polynomial is $1$
\cite{Stanley1996Combinatorics-a}, thus we only need to show that Equation
\eqref{eq:d2e3} is positive for $0 \leq p \leq n-2$.  It is sufficient to show
that the square-bracketed term of~\eqref{eq:d2e3},  
\begin{equation} \label{eq:d2e4}
2^{n-p-1} \sum_{\substack{j_p=1 + j_{p-1}}}^{n-1} \frac{1}{n-j_p} - \left( 1 + \frac{1}{n-1} \right),
\end{equation}
is positive for $0 \leq p \leq n-2$ and all $1 \leq j_1 \leq \cdots \leq j_{p-1} \leq n-2$. We can see that $\sum_{\substack{j_p=1 + j_{p-1}}}^{n-1} \frac{1}{n-j_p} \geq 1$. Moreover, since $0 \leq p \leq n-2$ then $2^{n-p-1} \geq 2$ and $1 + \frac{1}{n-1} \leq 2$ since $n \geq 2$,  proving the result.
\end{proof}

To present our results about the $h^*$-vector we begin explaining the
details with the following numbers introduced in
\cite{Katzman:math0408038}, which we refer to as the \emph{Katzman
coefficients}:
\begin{definition} \label{def:kat}
    For any positive integers $n$ and $r$ define the coefficients $A_i^{n,r}$ by
    $$ \sum_{i=0}^{n(r-1)} A_i^{n,r}T^i = (1 + T + \cdots + T^{r-1})^n.$$
    We also define the vector $\ve A_\cdot^{n,r}$ as 
    $\left( A_0^{n,r},A_1^{n,r},\ldots ,A_{n(r-1)}^{n,r} \right)$.
\end{definition}

Looking at the definition of the Katzman coefficient, we see that
$A_j^{n,2} = {n \choose j}$ and $A_j^{n,1} = 0$ unless $j=0$, in which
case we have $A_0^{n,1} = 1$. \\ \indent Below we derive some new and
useful equalities and prove symmetry and unimodality for the Katzman
coefficients. Katzman \cite{Katzman:math0408038} gave an explicit
equations for the $h^*$-vector of uniform matroid polytopes and the
coefficients of their Ehrhart polynomials, although he did not use the same
language. We restate it here for our purposes:

\begin{lemma}[See Corollary 2.9 in \cite{Katzman:math0408038}] \label{katlemma}
    Let $\Po(U^{r,n})$ be the matroid polytope 
    of the uniform matroid of rank $r$ on $n$ elements. Then 
    the $h^*$-polynomial of $\Po(U^{r,n})$ is given by

    \begin{equation} \label{eq:hvec2}
        \sum_{s=0}^{r-1} \sum_{j=0}^s \sum_{k=0}^j \sum_{l \geq k} (-1)^{s +j +k} \left[ {n \choose s} {s \choose j} {j \choose k} A_{(l-k)(r-s)}^{n-j,r-s} \right] T^{l}.
    \end{equation}
    That is, for $\ve h^*(\Po(U^{n,r)})) = (h^*_0,\ldots, h^*_r)$, 
    \begin{equation*} \label{eq:hvec2coeff}
        h^*_l = \sum_{s=0}^{r-1} \sum_{j=0}^s \sum_{k=0}^j(-1)^{s +j +k} {n \choose s} {s \choose j} {j \choose k} A_{(l-k)(r-s)}^{n-j,r-s}.
    \end{equation*} 
    For $r=2$ the $h^*$-polynomial of $\Po(U)^{n,2}$ is
    \begin{equation} \label{eq:dim2}
        \left( \sum_{l \geq 0} {n \choose 2l} T^l \right) - nT .
    \end{equation}
\end{lemma}


The
following lemma is a direct consequence of Corollary 2.9 in
\cite{Katzman:math0408038}:

\begin{lemma}
    Let $\Po(U^{2,n})$ be the matroid polytope 
    of the uniform matroid of rank $2$ on $n$ elements. 
    The $h^*$-vector of $\Po(U^{2,n})$ is unimodal.
\end{lemma}


The rank 2 case is an interesting example already. Although the 
$h^*$-vector is unimodal, it is not always symmetric. Next we
present some useful lemmas, the first a combinatorial
description of the Katzman coefficients.

\begin{lemma}
  For $i=0,\dots,n(r-1)$ we have
    \begin{equation} \label{eq:katcomb}
        A_i^{n,r} =\sum_{\substack{0a_0 + 1a_1 + \dots + (r-1)a_{r-1} = i\\ a_0 + a_1 + \dots + a_{r-1} = n }} { n \choose a_0,a_1, \cdots ,a_{r-1}} 
    \end{equation}
    where $a_0,\dots,a_{r-1}$ run through non-negative integers.
\end{lemma}

\begin{proof}
    Using the multinomial formula \cite{Stanley1997Enumerative-Com} we have

    \begin{align*}
     \sum_{i=0}^{n(r-1)} A_i^{n,r}T^i = & (1 + T + \cdots + T^{r-1})^n  \\
    = &\sum_{\substack{ a_0 + a_1 + \dots + a_{r-1} = n}} { n \choose a_0,a_1, \cdots ,a_{r-1}} 1^{a_0}T^{a_1}T^{2a_2}\dots T^{(r-1)a_{r-1}} \\
    = &\sum_{\substack{ a_0 + a_1 + \dots + a_{r-1} = n}} { n \choose a_0,a_1, \cdots ,a_{r-1}} T^{0a_0 + 1a_1 + \cdots + (r-1)a_{r-1}}.
    \end{align*}
    By grouping the powers of $T$ we get equation \eqref{eq:katcomb}.
\end{proof}

Next we present a generalization of a property of the binomial
coefficients. The following lemma relates the Katzman coefficients to
Katzman coefficients with one less element.
\begin{lemma} \label{lem:dimrel}
    \begin{equation} \label{eq:dimrel}
        A_i^{n,r} = \sum_{k=i-r+1}^i A_k^{n-1,r}
    \end{equation}
    where we define $A_p^{n-1,r} := 0$ when $p < 0$ or $p > (n-1)(r-1)$.
\end{lemma}

\begin{proof}
    \begin{equation*}
    \begin{split}
    & \sum_{i = 0}^{n(r-1)} A_i^{n,r} T^i  \\
    & = (1 + T + \dots + T^{r-1})^n = (1 + T + \dots + T^{r-1})^{n-1}(1 + T + \dots + T^{r-1})   \\
    & = \left( \sum_{i = 0}^{(n-1)(r-1)} A_i^{n-1,r} T^i \right)  (1 + T + \dots + T^{r-1})  \\
    & = \sum_{i = 0}^{(n-1)(r-1)} A_i^{n-1,r} T^i + \sum_{i = 0}^{(n-1)(r-1)} A_i^{n-1,r} T^{i+1} + \cdots + \sum_{i = 0}^{(n-1)(r-1)} A_i^{n-1,r} T^{i + r-1} \\
    & = \sum_{i = 0}^{(n-1)(r-1)} A_i^{n-1,r} T^i + \sum_{i = 1}^{(n-1)(r-1) +1} A_{i-1}^{n-1,r} T^{i} + \cdots + \sum_{i = r-1}^{(n-1)(r-1) + r-1} A_{i-r+1}^{n-1,r} T^{i} \\
    & = \sum_{i=0}^{(n-1)(r-1)+r-1} \left( A_i^{n-1,r} + A_{i-1}^{n-1,r} + \cdots + A_{i-r+1}^{n-1,r} \right) T^i.
    \end{split}
    \end{equation*}
    Thus we get equation \eqref{eq:dimrel}.
\end{proof}

The following lemma relates the Katzman coefficients of rank~$r$
with those of rank~$r-1$.
\begin{lemma} \label{lem:rankrel}
    \begin{displaymath}
    \sum A_i^{n,r}T^i = \sum_{k=0}^n { n \choose k} T^k \left( \sum_{l=0}^{k(r-2)} A_l^{k,r-1}T^l \right)
    \end{displaymath}
    or in other words
    \begin{equation*} \label{eq:rankrel}
    A_i^{n,r} \; = \; \sum_{\substack{k+l=i \\ 0 \leq k \leq n \\ 0 \leq l \leq k(r-2)}} {n \choose k}  A_l^{k,r-1}.
    \end{equation*}
\end{lemma}

\begin{proof}
    From Definition \ref{def:kat}
    \begin{align*}
        \sum_{i=0}^{n(r-1)} A_i^{n,r}T^i = & (1 + T + \cdots + T^{r-1})^n = (1 + \big[T + \cdots + T^{r-1}\big])^n \\
        = & \sum_{k=0}^n { n \choose k} \big[T + \cdots + T^{r-1}\big]^k = \sum_{k=0}^n { n \choose k} T^k\big[1 + \cdots + T^{r-2}\big]^k \\
        = & \sum_{k=0}^n { n \choose k} T^k \left( \sum_{l=0}^{k(r-2)} A_l^{k,r-1}T^l \right).
    \end{align*}
\end{proof}

\begin{lemma} \label{lem:katsymuni}
    The Katzman coefficients are unimodal and symmetric in the index $i$. That is, the vector $\left( A_0^{n,r}, A_1^{n,r}, \dots, A_{n(r-1)}^{n,r} \right)$ is unimodal and symmetric.
\end{lemma}

\begin{proof}
    We first prove symmetry. Considering equation \eqref{eq:katcomb} we assume that
    \begin{equation*} 
        0a_0 + 1a_1 + \dots + (r-1)a_{r-1} = i \qquad \text{and} \qquad a_0 + a_1 + \dots + a_{r-1} = n. 
    \end{equation*}
    These two assumptions imply that
    \begin{align*}
    &\big((r-1)-0\big)a_0 + \big((r-1)-1\big)a_1 + \dots + \big((r-1)-(r-1)\big)a_{r-1}  \\
     = & (r-1)a_0 + (r-1)a_1 + \dots + (r-1)a_n -0a_0 - 1a_1 - \dots - (r-1)a_{r-1} \\ 
     = & (r-1)n - i.
    \end{align*}
    Therefore
    \begin{align*}
      A_i^{n,r} = & \sum_{\substack{0a_0 + 1a_1 + \dots + (r-1)a_{r-1} = i\\ a_0 + a_1 + \dots + a_{r-1} = n }} { n \choose a_0,a_1, \cdots ,a_{r-1}} & \\
     = & \sum_{\substack{(r-1)a_0 + (r-2)a_1 + \dots + 0a_{r-1} = i\\ a_0 + a_1 + \dots + a_{r-1} = n }} { n \choose a_0,a_1, \cdots ,a_{r-1}} = A_{n(r-1)-i}^{n,r}.
    \end{align*}


    To prove unimodality we proceed by induction on $n$, where $r$ is fixed.
    First, $\ve A_\cdot^{1,r}$ is unimodal. Assume for $n-1$ that $\ve
    A_\cdot^{n-1,r} = \left(
            A_0^{n-1,r},A_1^{n-1,r},\cdots,A_{(n-1)(r-1)}^{n-1,r}\right)$ is
    unimodal. Using equation \eqref{eq:dimrel} and the fact that $\ve A_\cdot^{n-1,r}$
    is symmetric we get that $\ve A_\cdot^{n,r}$ is unimodal. To help see this,
    one can view equation \eqref{eq:dimrel} as a sliding window over $r$ elements of
    the vector $\ve A_\cdot^{n-1,r}$, that is, $A_i^{n,r}$ is equal to the sum
    of the $r$ elements in a window over the vector $\ve A_\cdot^{n-1,r}$. As
    the window slides up the vector $\ve A_\cdot^{n-1,r}$, the sum will
    increase.  When the window is on the center of $\ve A_\cdot^{n-1,r}$
    symmetry and unimodality of $\ve A_\cdot^{n-1,r}$ will imply unimodality of
    $\ve A_\cdot^{n,r}$.
\end{proof}

Now we use the explicit equation for the $h^*$-vector of uniform
matroid polytopes to prove partial unimodality of rank $3$ uniform
matroids. First we note that the coefficient of $T^l$, $h^*_l$, in
equation \eqref{eq:hvec2} is
\begin{displaymath}
h^*_l = \sum_{s=0}^{r-1} \sum_{j=0}^s \sum_{k=0}^j(-1)^{s +j +k} {n \choose s} {s \choose j} {j \choose k} A_{(p-k)(r-s)}^{n-j,r-s}.
\end{displaymath}

Letting the rank $r = 3$, and using equation \eqref{eq:hvec2}, we 
get the $h^*$-polynomial (which is grouped by values of $s$ 
from \eqref{eq:hvec2}),
\begin{multline*}
    \sum_{l \geq 0} \left[ {n \choose 0 } A_{3l}^{n,3} + {n \choose 1} \left( -A_{2l}^{n,2} + A_{2l}^{n-1,2} - A_{2(l-1)}^{n-1,2} \right) + \right.  \\
  + \left.  {n \choose 2} \left( A_{l}^{n,1} - 2A_{l}^{n-1,1} + 2A_{l-1}^{n-1,1} + A_{l}^{n-2,1} - 2A_{l-1}^{n-2,1} + A_{l-2}^{n-2,1} \right) \right] T^l.
\end{multline*}
Now using that $A_{i}^{n,2} = { n \choose i}$ and $A_{i}^{n,1} = \delta_0(i)$, where
$\delta_j(p) = \left\{
\begin{array}{lll}
1 & \text{if } p = j \\
0 & \text{else}
\end{array}
\right.$ ,
 we get
\begin{multline*}
    \sum_{l \geq 0}  \left[ A_{3l}^{n,3} + n \left( - { n \choose 2l} \right. \left. + {n-1 \choose 2l} -  {n-1 \choose 2l - 2} \right) + \right.  \\
    \left. + {n \choose 2} \left(  \delta_0(l) - 2 \delta_0(l) + 2 \delta_{1}(l) + \delta_0(l) - 2 \delta_{1}(l) + \delta_2(l)   \right) \right] T^l
\end{multline*}
\begin{equation*}
    = \sum_{l \geq 0}  \left[ A_{3l}^{n,3} + n \left( - { n \choose 2l} + {n-1 \choose 2l} - {n-1 \choose 2l - 2} \right) + \delta_2(l) {n \choose 2}  \right] T^l.
\end{equation*}
Using properties of the binomial coefficients, we see that
\begin{align*}
     \left[ -{n \choose 2l} \right]  + {n-1 \choose 2l} - {n-1 \choose 2l - 2} = & \left[ -{n-1 \choose 2l} - {n-1 \choose 2l -1} \right] + {n-1 \choose 2l} - {n-1 \choose 2l - 2} \\
= & - {n-1 \choose 2l -1} - {n-1 \choose 2l - 2} \\
= & -{n \choose 2l - 1}.
\end{align*}
So the $h^*$-polynomial of rank three uniform matroid polytopes is
\begin{equation} \label{hvecrank3}
     \sum_{l \geq 0}  \left[ A_{3l}^{n,3} - n{n \choose 2l - 1}  + \delta_2(l) {n \choose 2}  \right] T^l.
\end{equation}
Using Lemma \ref{eq:dimrel}, the coefficient of $T^l$, if $3l \leq n$, is
\begin{equation} \label{hvecrank3eq2}
    h^*_l = \sum_{\substack{k+p = 3l \\ 0 \leq p \leq k \leq n}} {n \choose k}{k \choose p} - n{n \choose 2l-1} + \delta_2(l){n \choose 2}
\end{equation}
\begin{multline*}
= \left[ {n \choose 3l}{3l \choose 0} + {n \choose 3l-1}{ 3l-1 \choose 1} + \cdots+ {n \choose 3l - \lfloor 3l/2 \rfloor}   {3l - \lfloor 3l/2 \rfloor \choose \lfloor 3l/2 \rfloor} \right] + \\
   - n{n \choose 2l-1} + \delta_2(l){n \choose 2}.
\end{multline*}

Next we show that when $g$ is fixed $A_g^{n,3}$ is a polynomial of degree $g$
in the indeterminate $n$, with positive leading coefficient. Assume  $g \leq
n$. Considering Lemma \ref{lem:rankrel} and when $g \leq n$, 
\begin{equation}
   A_{g}^{n,3} = \sum_{\substack{k+p = g \\ 0 \leq p \leq k \leq n}} {n \choose k}{k \choose p}
   \label{katdim3}
\end{equation}
where $n \choose q$ is a polynomial
of degree $q$ with positive leading coefficient. The highest
degree polynomial in the sum is $n \choose g$, a degree $g$ 
polynomial. Hence $A_g^{n,3}$ is a polynomial of degree $g$ 
in the indeterminate $n$, with positive leading coefficient. 
If $g \geq n$, then $A_g^{n,3} = A_{n-g}^{n,3}$ since
the Katzman coefficients are symmetric by Lemma \ref{lem:katsymuni}.

\begin{proof}[Proof of Theorem \ref{partialuni} Part (2)]
     Let $I$ be a non-negative integer. From above we see that $A_g^{n,3}$ is a
     degree $g$ polynomial in the indeterminate $n$, with positive leading
     coefficient.  Equation \eqref{hvecrank3} is the $h^*$-polynomial of
     $U^{3,n}$, which is a sum of polynomials in $n$, the highest degree
     polynomial being $A_{3l}^{n,3}$, a polynomial of degree $3l$. So, $h^*_l -
     h^*_{l-1}$ is the difference of a degree $3l$ and $3(l-1)$ polynomial,
     hence $h^*_l - h^*_{l-1}$ is a degree $3l$ polynomial with positive
     leading coefficient.  For sufficiently large $n$, call it $n(I)$, $h^*_l
     - h^*_{l-1}$ is positive for $0 \leq l \leq I$. Hence, the $h^*$-vector of
     $U^{3,n}$ is non-decreasing up to the index $I$ for $n \geq n(I)$.
\end{proof}

One might ask if equation \eqref{katdim3} has a simpler form. We ran the
\emph{WZ} algorithm on our expression, which proved that equation
\eqref{katdim3} can not be written as a linear combination of a fixed number of
hypergeometric terms (\emph{closed form}) \cite{Petkovsek1996AB}. There is
still the possibility that this expression has a simpler form, though not a
closed form as described above.

\section*{Acknowledgments}
The first author was supported by NSF grant DMS-0608785. The second
author was supported by VIGRE-GRANT DMS-0135345 and DMS-0636297. The
third author was supported by a 2006/2007 Feodor Lynen Research
Fellowship from the Alexander von Humboldt Foundation. We thank Dillon
Mayhew and Gordon Royle for providing data for the matroids listed in
Table \ref{tab:hvec}.

\bibliographystyle{abbrvnat}
\bibliography{barvinok,matroid}

\begin{thebibliography}{41}
\expandafter\ifx\csname natexlab\endcsname\relax\def\natexlab#1{#1}\fi
\expandafter\ifx\csname url\endcsname\relax
  \def\url#1{{\tt #1}}\fi

\bibitem[Barvinok(1994)]{bar}
A.~I. Barvinok.
\newblock Polynomial time algorithm for counting integral points in polyhedra
  when the dimension is fixed.
\newblock {\em Mathematics of Operations Research}, 19:\penalty0 769--779,
  1994.

\bibitem[Barvinok(2006)]{barvinok-2006-ehrhart-quasipolynomial}
A.~I. Barvinok.
\newblock Computing the {E}hrhart quasi-polynomial of a rational simplex.
\newblock {\em Math. Comp.}, 75\penalty0 (255):\penalty0 1449--1466
  (electronic), 2006.

\bibitem[Barvinok and Pommersheim(1999)]{barvinok:99}
A.~I. Barvinok and J.~E. Pommersheim.
\newblock An algorithmic theory of lattice points in polyhedra.
\newblock In L.~J. Billera, A.~Bj\"orner, C.~Greene, R.~E. Simion, and R.~P.
  Stanley, editors, {\em New Perspectives in Algebraic Combinatorics},
  volume~38 of {\em Math. Sci. Res. Inst. Publ.}, pages 91--147. Cambridge
  Univ. Press, Cambridge, 1999.

\bibitem[Barvinok and Woods(2003)]{barvinok-woods-2003}
A.~I. Barvinok and K.~Woods.
\newblock Short rational generating functions for lattice point problems.
\newblock {\em Journal of the AMS}, 16\penalty0 (4):\penalty0 957--979, 2003.

\bibitem[Beck et~al.(2006)Beck, Haase, and
  Sottile]{beck-haase-sottile:theorema}
M.~Beck, C.~Haase, and F.~Sottile.
\newblock {Formulas of Brion, Lawrence, and Varchenko on rational generating
  functions for cones}.
\newblock eprint arXiv:{\penalty0}math.CO/0506466, 2006.

\bibitem[Beck and Sottile(2007)]{beck-sottile:irrational}
M.~Beck and F.~Sottile.
\newblock Irrational proofs for three theorems of {S}tanley.
\newblock {\em European Journal of Combinatorics}, 28\penalty0 (1):\penalty0
  403--409, 2007.

\bibitem[Billera et~al.(2006)Billera, Jia, and Reiner]{billera-2006}
L.~J. Billera, N.~Jia, and V.~Reiner.
\newblock A quasisymmetric function for matroids.
\newblock eprint arXiv:math/0606646, 2006.

\bibitem[Brightwell and Winkler(1991)]{brightwellwinkler91}
G.~Brightwell and P.~Winkler.
\newblock Counting linear extensions.
\newblock {\em Order}, 8\penalty0 (3):\penalty0 225--242, 1991.

\bibitem[Brion(1988)]{Brion88}
M.~Brion.
\newblock Points entiers dans les poly{\'e}dres convexes.
\newblock {\em Ann. Sci. {\'E}cole Norm. Sup.}, 21\penalty0 (4):\penalty0
  653--663, 1988.

\bibitem[De~Loera et~al.(2004{\natexlab{a}})De~Loera, Haws, Hemmecke, Huggins,
  Sturmfels, and Yoshida]{latte2}
J.~A. De~Loera, D.~Haws, R.~Hemmecke, P.~Huggins, B.~Sturmfels, and R.~Yoshida.
\newblock Short rational functions for toric algebra and applications.
\newblock {\em Journal of Symbolic Computation}, 38\penalty0 (2):\penalty0
  959--973, 2004{\natexlab{a}}.

\bibitem[De~Loera et~al.(2005)De~Loera, Haws, Hemmecke, Huggins, Tauzer, and
  Yoshida]{lattesoft}
J.~A. De~Loera, D.~C. Haws, R.~Hemmecke, P.~Huggins, J.~Tauzer, and R.~Yoshida.
\newblock Software and user's guide for latte v.1.1, 2005.
\newblock URL \url{www.math.ucdavis.edu/~latte}.

\bibitem[De~Loera et~al.(2007)De~Loera, Haws, and K\"oppe]{HawsMatroid-Polytop}
J.~A. De~Loera, D.~C. Haws, and M.~K\"oppe.
\newblock Matroid polytopes, 2007.
\newblock URL \url{http://math.ucdavis.edu/~haws/Matroids/}.

\bibitem[De~Loera et~al.(2004{\natexlab{b}})De~Loera, Hemmecke, Tauzer, and
  Yoshida]{latte1}
J.~A. De~Loera, R.~Hemmecke, J.~Tauzer, and R.~Yoshida.
\newblock Effective lattice point counting in rational convex polytopes.
\newblock {\em Journal of Symbolic Computation}, 38\penalty0 (4):\penalty0
  1273--1302, 2004{\natexlab{b}}.

\bibitem[De~Loera et~al.(2006)De~Loera, Rambau, and
  Santos]{De-Loera2006Triangulations}
J.~A. De~Loera, J.~Rambau, and F.~Santos.
\newblock Triangulations: Applications, structures, algorithms.
\newblock Book manuscript, 2006.

\bibitem[De~Negri and Hibi(1997)]{Negri1997Gorenstein-Alge}
E.~De~Negri and T.~Hibi.
\newblock Gorenstein algebras of {V}eronese type.
\newblock {\em Journal of Algebra}, 193\penalty0 (2):\penalty0 629--639, July
  1997.

\bibitem[Dyer and Frieze(1988)]{dyerfrieze88}
M.~E. Dyer and A.~M. Frieze.
\newblock On the complexity of computing the volume of a polyhedron.
\newblock {\em SIAM J. Comput.}, 17\penalty0 (5):\penalty0 967--974, 1988.
\newblock ISSN 0097-5397.

\bibitem[Edmonds(2003)]{Edmonds2003Submodular-func}
J.~Edmonds.
\newblock Submodular functions, matroids, and certain polyhedra.
\newblock In M.~J\"unger, G.~Reinelt, and G.~Rinaldi, editors, {\em
  Combinatorial Optimization -- Eureka, You Shrink!: Papers Dedicated to Jack
  Edmonds. 5th International Workshop, Aussois, France, March 5--9, 2001,
  Revised Papers}, volume 2570 of {\em Lecture notes in computer science},
  pages 11--26. Springer-Verlag, Berlin, 2003.

\bibitem[Elekes(1986)]{elekes86}
G.~Elekes.
\newblock A geometric inequality and the complexity of computing volume.
\newblock {\em Discrete Comput. Geom.}, 1\penalty0 (4):\penalty0 289--292,
  1986.

\bibitem[Feichtner and Sturmfels(2005)]{Feichtner2004Matroid-polytop}
E.~M. Feichtner and B.~Sturmfels.
\newblock Matroid polytopes, nested sets and {B}ergman fans.
\newblock {\em Port. Math.}, 62:\penalty0 437--468, 2005.

\bibitem[Fukuda(2006)]{cdd}
K.~Fukuda.
\newblock cdd+, a {C++} implementation of the double description method of
  {Motzkin} et al.
\newblock Available from URL
  {\url{http://www.ifor.math.ethz.ch/~fukuda/cdd_home/cdd.html}}, 2006.

\bibitem[Gelfand et~al.(1987)Gelfand, Goresky, MacPherson, and
  Serganova]{Gelfand1987Combinatorial-g}
I.~M. Gelfand, M.~Goresky, R.~D. MacPherson, and V.~V. Serganova.
\newblock Combinatorial geometries, convex polyhedra, and {S}chubert cells.
\newblock {\em Adv. Math.}, 63:\penalty0 301--316, 1987.

\bibitem[Goodman and O'Rourke(1997)]{1997Handbook-of-dis}
J.~E. Goodman and J.~O'Rourke, editors.
\newblock {\em Handbook of Discrete and Computational Geometry}.
\newblock CRC Press, Inc., Boca Raton, FL, USA, 1997.

\bibitem[Hibi(1992)]{Hibi1992Algebraic-Combi}
T.~Hibi.
\newblock {\em Algebraic Combinatorics on Convex Polytopes}.
\newblock Carslaw Publications, Glebe, Australia, 1992.

\bibitem[Katzman(2005)]{Katzman:math0408038}
M.~Katzman.
\newblock The {H}ilbert series of algebras of {V}eronese type.
\newblock {\em Communications in Algebra}, 33:\penalty0 1141--1146, 2005.

\bibitem[Khachiyan(1993)]{khachiyan93}
L.~Khachiyan.
\newblock Complexity of polytope volume computation.
\newblock In {\em New Trends in Discrete and Computational Geometry}, volume~10
  of {\em Algorithms Combin.}, pages 91--101. Springer, Berlin, 1993.

\bibitem[K{\"o}ppe(2007{\natexlab{a}})]{latte-macchiato}
M.~K{\"o}ppe.
\newblock {LattE macchiato}, version 1.2-mk-0.9, an improved version of {De
  Loera} et al.'s {LattE} program for counting integer points in polyhedra with
  variants of {Barvinok}'s algorithm.
\newblock Available from URL
  {\url{http://www.math.uni-magdeburg.de/~mkoeppe/latte/}}, 2007{\natexlab{a}}.

\bibitem[K{\"o}ppe(2007{\natexlab{b}})]{koeppe:irrational-barvinok}
M.~K{\"o}ppe.
\newblock A primal {B}arvinok algorithm based on irrational decompositions.
\newblock {\em SIAM Journal on Discrete Mathematics}, 21\penalty0 (1):\penalty0
  220--236, 2007{\natexlab{b}}.

\bibitem[K\"oppe and Verdoolaege(2007)]{koeppe-verdoolaege:parametric}
M.~K\"oppe and S.~Verdoolaege.
\newblock Computing parametric rational generating functions with a primal
  {B}arvinok algorithm.
\newblock eprint arXiv:{\penalty0}0705.3651 [math.CO], 2007.

\bibitem[Lawrence(1991{\natexlab{a}})]{lawrence91}
J.~Lawrence.
\newblock Polytope volume computation.
\newblock {\em Math. Comp.}, 57\penalty0 (195):\penalty0 259--271,
  1991{\natexlab{a}}.

\bibitem[Lawrence(1991{\natexlab{b}})]{lawrence91-2}
J.~Lawrence.
\newblock Rational-function-valued valuations on polyhedra.
\newblock In {\em Discrete and Computational Geometry (New Brunswick, NJ,
  1989/1990)}, volume~6 of {\em DIMACS Ser. Discrete Math. Theoret. Comput.
  Sci.}, pages 199--208. Amer. Math. Soc., Providence, RI, 1991{\natexlab{b}}.

\bibitem[Oxley(1992)]{Oxley1992Matroid-Theory}
J.~Oxley.
\newblock {\em Matroid Theory}.
\newblock Oxford University Press, New York, NY, USA, 1992.

\bibitem[{Petkov\v sek} et~al.(1996){Petkov\v sek}, Wilf, and
  Zeilberger]{Petkovsek1996AB}
M.~{Petkov\v sek}, H.~S. Wilf, and D.~Zeilberger.
\newblock {\em ${A}={B}$}.
\newblock AK Peters, Ltd., 1996.
\newblock URL \url{http://www.math.upenn.edu/~wilf/AeqB.html}.

\bibitem[Schrijver(1986)]{Schrijver1986Theory-of-linea}
A.~Schrijver.
\newblock {\em Theory of Linear and Integer Programming}.
\newblock John Wiley \& Sons, Inc., New York, NY, USA, 1986.

\bibitem[Schrijver(2003)]{Schrijver2003Combinatorial-O}
A.~Schrijver.
\newblock {\em Combinatorial Optimization: Polyhedra and Efficiency}.
\newblock Springer, 2003.

\bibitem[Speyer(2006)]{SpeyerA-matroid-invar}
D.~E. Speyer.
\newblock A matroid invariant via the {K}-theory of the {G}rassmannian.
\newblock eprint arXiv:math/0603551v1, 2006.

\bibitem[Stanley(1996)]{Stanley1996Combinatorics-a}
R.~P. Stanley.
\newblock {\em Combinatorics and Commutative Algebra: Second Edition}.
\newblock Birkh\"auser, Boston, 2nd edition, 1996.

\bibitem[Stanley(1997)]{Stanley1997Enumerative-Com}
R.~P. Stanley.
\newblock {\em Enumerative Combinatorics}, volume~1.
\newblock Cambridge University Press, 1997.

\bibitem[Topkis(1984)]{Topkis1984Adjacency-on-Po}
D.~M. Topkis.
\newblock Adjacency on polymatroids.
\newblock {\em Mathematical Programming}, 30\penalty0 (2):\penalty0 229--237,
  October 1984.

\bibitem[Verdoolaege and Woods(2006)]{verdoolaege-woods-2005}
S.~Verdoolaege and K.~M. Woods.
\newblock {Counting with rational generating functions}.
\newblock eprint arXiv:{\penalty0}math.CO/0504059, May 2006.
\newblock To appear in The Journal of Symbolic Computation.

\bibitem[Welsh(1976)]{Welsh1976Matroid-Theory}
D.~Welsh.
\newblock {\em Matroid Theory}.
\newblock Academic Press, Inc., 1976.

\bibitem[Woods(2004)]{Woods:thesis}
K.~Woods.
\newblock {\em Rational Generating Functions and Lattice Point Sets}.
\newblock PhD thesis, University of Michigan, 2004.

\end{thebibliography}

\end{document}